\documentclass{amsart}

\usepackage{amssymb,color,epsf}
\usepackage{mathtools}
\usepackage{enumerate}

\ifx\pdfpageheight\undefined
   \usepackage[dvips,colorlinks=true,linkcolor=blue,citecolor=red,%
      urlcolor=green]{hyperref}
   \usepackage[dvips]{graphicx}
   \makeatletter
   \edef\Gin@extensions{\Gin@extensions,.mps}
   \makeatother
\else
 \usepackage[pdftex]{graphicx}
   \usepackage[bookmarksopen=false,pdftex=true,breaklinks=true,%
      backref=page,pagebackref=true,plainpages=false,%
      hyperindex=true,pdfstartview=FitH,colorlinks=true,%
      pdfpagelabels=true,colorlinks=true,linkcolor=blue,%
      citecolor=red,urlcolor=green,hypertexnames=false%
      ]%
   {hyperref}
\fi
\usepackage{bm}
\usepackage{amsmath}
\usepackage{tabularx}
\usepackage{color, colortbl}
\usepackage{url}

\DeclareMathAlphabet{\mathpzc}{OT1}{pzc}{m}{it}
\usepackage[margin=1.5in]{geometry}

\usepackage [cmtip,arrow]{xy}
\xyoption{all}
\usepackage {pb-diagram,pb-xy}
\usepackage[mathscr]{eucal}
\usepackage[utf8]{inputenc}
\usepackage[T1]{fontenc}
\usepackage[english]{babel}
\usepackage{bm}
\usepackage{tabularx}
\usepackage{ltablex}

\usepackage{bm}

\usepackage{amsmath,amssymb,amsfonts}
\newtheorem{theorem}{Theorem}[section]
\newtheorem{lemma}[theorem]{Lemma}

\newtheorem{proposition}[theorem]{Proposition}
\newtheorem{conjecture}[theorem]{Conjecture}

\newtheorem{question}[theorem]{Question}

\theoremstyle{definition}
\newtheorem{definition}[theorem]{Definition}
\newtheorem{example}[theorem]{Example}

\newtheorem{notation}[theorem]{Notation}

\theoremstyle{remark}
\newtheorem{remark}[theorem]{Remark}

\definecolor{DarkBlue}{rgb}{0,0.1,0.55}

\numberwithin{equation}{section}

\newcommand {\hide}[1]{}

\newcommand {\junk}[1]{}

\newcommand {\R} {\mathbb{R}}

\newcommand {\C}     {\mathbb{C}}

\newcommand {\Z}  {\mathbb{Z}}

\newcommand {\ZZ} {{\rm Z}}
\newcommand {\RR} {{\mathcal R}}

\newcommand {\eps} {{\varepsilon}}

\newcommand {\PP}     {\mathbb{P}} 

\newcommand{\card}{\mathrm{card}}

\def\addots{\mathinner{\mkern1mu
\raise1pt\vbox{\kern7pt\hbox{.}}
\mkern2mu\raise4pt\hbox{.}\mkern2mu
\raise7pt\hbox{.}\mkern1mu}}

\newcommand{\HH}  {\mbox{\rm H}}

\newcommand{\Hom}{\mathrm{Hom}}

\newcommand\colim{\op{colim}}

\newcommand{\Ob}{\mathrm{Ob}}

\newcommand{\Top}{\mathbf{Top}}

\newcommand{\nc}{\newcommand}
\newcommand{\rc}{\renewcommand}
\nc{\mc}{\mathcal}
\rc{\t}{\text}
\nc{\op}[1]{\operatorname{#1}}
\nc{\opcat}[1]{\mathbf{#1}}

\nc{\id}{\op{id}}

\nc{\umutnote}[1]{{\marginpar{\small \textcolor{blue}{#1}}}}

\def\sbullet{\raisebox{0.009cm}{\scalebox{0.5}{$\bullet$}}}
\def\arrowdiagram{\sbullet \to \sbullet}

\def\bigO{\op{O}}
\def\loccit{\emph{loc. cit.}}

\def\Bigoplus{\displaystyle\bigoplus}

\nc{\cA}{\mc{A}}\nc{\cB}{\mc{B}}\nc{\cC}{\mc{C}}\nc{\cD}{\mc{D}}\nc{\cE}{\mc{E}}\nc{\cF}{\mc{F}}\nc{\cG}{\mc{G}}\nc{\cH}{\mc{H}}\nc{\cI}{\mc{I}}\nc{\cJ}{\mc{J}}\nc{\cK}{\mc{K}}\nc{\cL}{\mc{L}}\nc{\cM}{\mc{M}}\nc{\cN}{\mc{N}}\nc{\cO}{\mc{O}}\nc{\cP}{\mc{P}}\nc{\cQ}{\mc{Q}}\nc{\cR}{\mc{R}}\nc{\cS}{\mc{S}}\nc{\cT}{\mc{T}}\nc{\cU}{\mc{U}}\nc{\cV}{\mc{V}}\nc{\cW}{\mc{W}}\nc{\cX}{\mc{X}}\nc{\cY}{\mc{Y}}\nc{\cZ}{\mc{Z}}

\rc{\PP}{\mathbb{P}}
\rc{\AA}{\mathbb{A}}
\rc{\ZZ}{\mathbb{Z}}
\rc{\RR}{\mathbb{R}}
\nc{\bbC}{\mathbb{C}}
\nc{\CC}{\mathbb{C}}

\nc{\code}[1]{{\texttt{#1}}}
\nc{\mcode}[1]{{\text{\texttt{#1}}}}
\nc{\xto}[1]{\raisebox{-0.03cm}{\scalebox{0.85}{$\,\xrightarrow{#1}\,$}}}
\nc{\xtonormal}[1]{\xrightarrow{#1}}
\nc{\xfrom}[1]{\xleftarrow{#1}}
\nc{\sidenote}[1]{\marginpar{\small #1}}

\nc{\Aff}{\opcat{Aff}}
\nc{\AffVar}{\opcat{AffVar}}
\nc{\ProjVar}{\opcat{ProjVar}}
\nc{\GAP}{\opcat{GrAlgPairs}}
\nc{\GA}{\opcat{GrAlg}}

\nc{\acc}{\mathrm{a.c.c}}
\nc{\GL}{\mathrm{GL}}

\nc{\Mod}{\t{-}\opcat{Mod}}
\nc{\Sub}{\opcat{Sub}}
\nc{\iso}{\cong}
\nc{\compose}{\circ}

\nc{\Nat}{\mathrm{Nat}}

\nc{\IC}{\mathrm{IFC}}

\begin{document}
\title[Categorical Complexity]
{Categorical Complexity}
\author{Saugata Basu}
\address{Department of Mathematics,
Purdue University, West Lafayette, IN 47906, U.S.A.}
\email{sbasu@math.purdue.edu}
\thanks{The first author was supported in part by NSF grants CCF-1319080, CCF-1618981,  DMS-1620271 and CCF-1910441 while working on this paper.}
\author{M. Umut Isik}
\address{Department of Mathematics, University of California, Irvine, Irvine, CA 92697}
\email{isik@math.uci.edu}

\subjclass{Primary 18A10, 68Q15; Secondary 14Q20}


\begin{abstract}
We introduce a notion of complexity of diagrams (and in particular of objects and morphisms) in an arbitrary category, as well as a notion of complexity of functors between categories equipped with complexity functions. We discuss several examples of this new definition in categories of wide common interest, such as finite sets, Boolean functions, topological spaces, vector spaces, semi-linear and semi-algebraic sets,  graded algebras, affine and projective varieties and schemes, and modules over polynomial rings. We show that on one hand categorical complexity recovers in several settings classical notions of non-uniform computational complexity (such as circuit complexity), while on the other hand it has features which make it mathematically more natural.  We also postulate that studying functor complexity is the categorical analog of classical questions in complexity theory about separating different complexity classes.  
\end{abstract}

\maketitle
\tableofcontents
\section{Introduction}
\label{sec:intro}

It is usual to associate some measure of complexity to mathematical objects.
For example, the complexity of a polynomial is often measured by its degree,  or alternatively by the volume of its Newton polytope, or the
number the monomials appearing in it with non-zero coefficients, or the least number of operations needed for an algorithm to evaluate the 
polynomial at a given point. Once a notion of complexity is fixed, one can make quantitative statements
about properties of the objects in terms of their complexity. In the case of polynomials for example, there are many results on upper bounds on
the topological invariants of the variety that the polynomial defines (see for example the survey \cite{BPR10}), 
the number of steps needed to desingularize the variety \cite{BGMW2011},
and many other functions defined on the space of polynomials, in terms of the chosen complexity measure.

The notion of complexity also arose in theoretical computer science
as a means of studying
efficiency of algorithms and also to measure the intrinsic hardness  of certain algorithmic problems. The latter led to the development of structural complexity theory and in particular to the famous P versus NP questions for
discrete complexity classes that remains unresolved until today. 
Even though these arose first in the context of decision problems and Boolean functions, there have been subsequent attempts to generalize the scope of computational complexity to other classes of mathematical objects --  for example, the Blum-Shub-Smale (B-S-S) theory for computations over reals as well as complex numbers \cite{BSS}, over more general structures \cite{Poizat}, for polynomials \cite{ValiantClasses, ValiantPermanent, vonzurgathen87}, 
for constructible sheaves and functions \cite{BasuConstr}. Some of these generalizations are motivated by costs of computations in certain models 
of computations, while others by the desire to have a sound internal notion of complexity for 
mathematical objects. Remarkably,
there exists analogues of the P versus NP question in all the generalizations mentioned above. Thus, it seems that there should be a more fundamental way of looking at questions arising in  computational complexity 
theory
which unifies these various viewpoints.

The goal of the current paper is to develop this general theory of complexity  via  a categorical approach that reconciles the intuitive notion of complexity of mathematical objects with the different notions of computational and circuit complexities used in theoretical computer science.

We start by defining a categorical notion called a \emph{diagram computation}. In an arbitrary category $\cC$ with a chosen set of morphisms called \emph{basic morphisms}, diagram computations can be used to construct diagrams, and in particular objects and morphisms in $\cC$ as follows. At the first level, one starts with a diagram consisting entirely of basic morphisms, and then successively adds limits and colimits of arbitrary 
subdiagrams, along with the accompanying morphisms from/to those subdiagrams, to construct more and more complex diagrams. Any diagram isomorphic to a subdiagram of the final resulting diagram is said to be computed by the diagram computation. This allows us to associate, to each object, arrow, or diagram in $\cC$, a complexity by counting the number of limits, colimits and basic morphisms used in its most efficient computation. Diagram computations come in three kinds, the full version described above, called a \emph{mixed computation}, and two more restricted ones where one is either allowed to use only the so-called constructive limits, or only constructive colimits; leading to the notions of \emph{mixed complexity}, \emph{limit complexity}, and \emph{colimit complexity} of objects, arrows or any diagrams in $\cC$.

While our notion of complexity bears some similarity with a more classical view of complexity coming from logic, namely descriptive complexity \cite{Immerman},
there is one important respect in which our notion of complexity differs significantly from 
all classical notions. In our categorical world, isomorphic objects should have identical  complexity  -- which is indeed the case with our definition. Thus, we are able to define a good notion of complexity in the category, say, of affine or projective schemes which is independent of embeddings. This is very natural from the mathematical point of view -- making our theory completely geometric in those settings -- but is sometimes at odds with ordinary complexity theory which deals with embedded objects (like subsets of the Boolean hypercube or subvarieties of $\C^n$). Nevertheless, we will show that, with the appropriate choice of category and basic morphisms, even non-categorical notions of complexity can be meaningfully embedded in categorical ones.

A fundamentally new point of view emerges when one thinks about the categorical analogues of classical complexity questions. For every functor between two categories for which complexity is defined, one can define a natural notion of complexity of  the functor. Unlike, the complexity of diagrams, which are numbers,
the complexity of a functor is a function $f:\mathbb{N}\rightarrow \mathbb{N}$, and one can ask whether the complexity of 
such a functor is bounded from above by a polynomial. In this way classical questions about separation of complexity classes become, in the categorical world, questions about polynomial boundedness of the complexity of certain natural functors. With this shift of viewpoint, one can pose many questions about complexities of functors which have no direct analogues in the world of computational complexity. Well-studied properties of functors, such as preservation of limits and colimits, adjointness, etc. are important from this point of view.

The importance of functor complexity was already suggested in \cite{BasuConstr}, where the complexities of adjoint pairs of functors between the categories of semi-algebraically constructible sheaves on finite dimensional real affine spaces were posited as 
generalizing the P versus NP question in the real B-S-S model.

We remark here that connections between computability and logic on one hand and
category theory and topos theory on the other hand has a long history (see for example, \cite{Lambek-Scot, MacLane-Moerdijk}). 
A more recent work on computability and complexity in categorical structures using the notion of Kan extensions appears in \cite{Yanofsky15}.
However, our goal is different, 
and it is to develop a completely general notion of complexity, based on category theory, that 
is useful in studying basic objects in algebra and geometry from a quantitative point of view.  To the best of our knowledge this task has not been undertaken before.

We now give a brief summary of our results. After the basic definitions in Section~\ref{sec:definitions}, we look at several key examples. 
For sets, the colimit complexity of a set $S$ is $\card(S) + 1$ 
(Proposition \ref{prop:sets}).
Infinite sets are ``non-computable''
in this theory. 
For topological spaces, colimit computations starting from simplices
and face maps define a simplicial complexity for topological spaces. Similarly, mixed computations
starting from points and intervals give rise to cubical complexity. These measure how hard it is to make a
space from simplices and cubes respectively. Another important example is one where we recover monotone Boolean complexity from the categorical complexity in the lattice of subsets of a finite set.

In order to relate the new notion of categorical complexity, with pre-existing notions of (non-uniform) complexity, such as circuit complexity, or lengths of straight-line programs (we refer the reader to the books \cite{Burgisser-book1,Burgisser-book2} for these notions), we prove certain comparison theorems.
The first set of such theorems are about affine varieties, affine schemes, and algebras over a field. We show that the affine zero-set of a polynomial with low arithmetic circuit complexity has low limit complexity (Theorem~\ref{thm_circuittodiagram}); on the other hand, if $X$ is a variety with low limit complexity, then it is \emph{isomorphic} to the zero-set of a polynomial 
with low arithmetic circuit complexity (Theorem~\ref{thm_diagramtocircuit}). The same results hold for affine schemes and algebras. For projective schemes in $\PP^{n}$: by building affine pieces with limits and then gluing them using colimits, we show that the mixed complexity of a projective scheme is bounded above by a constant multiple of $n^2N$, where $N$ is the arithmetic circuit complexity of its defining equations.  

The categorical complexities of isomorphic varieties are equal by definition, while circuit complexity, being a non-geometric attribute, does not share this property. In order to reconcile these two notions, we consider in Section~\ref{sec:comparisons} two additional categories where circuit complexity of polynomials is, in a sense, embedded into categorical complexity. The first of these is the category of pairs of graded algebras, constructed specifically to make this embedding possible. Still, it remains to be seen how complexity in this category compares to arithmetic circuit complexity of polynomials, or to the complexity of projective varieties discussed in \cite{Isik2019}. The second category considered here is the category of modules over polynomial rings, where we prove that the colimit complexities of a sequence of morphism diagram $\left(k\left[ x_1,\dots,x_n \right] \xto{1\mapsto f_n} k\left[ x_1,\dots,x_n \right]\right)_{n > 0}$ are bounded by a 
quasi-polynomial function of $n$, if and only if the arithmetic circuit complexities of the sequence $(f_n)_{n > 0}$ are also bounded by a quasi-polynomial function of $n$ (cf. Remark \ref{rem:RMod}).

In Section~\ref{sec:functors}, we discuss the behavior of categorical complexity under the action of  functors. Limit and colimit computations are preserved under right and left adjoints respectively. We define the complexity of a functor $F:\cC \to \cD$ as a function $C(F)(n)$ of $n$, where $C(F)(n)$ is the supremum of the complexity of $F(D)$, where $D$ runs over all diagrams in $\cC$ whose complexity is less than or equal to $n$. 
We posit that the question of whether the complexity of the ``image functor''
(see Definition \ref{def:images})
on the morphism category $\cC^{\arrowdiagram}$ is polynomially bounded,  is the categorical analogue of the P vs NP problem for the category $\cC$ (cf. discussion in the beginning of 
Section~\ref{subsec:complexity-of-image-functor}).
We investigate this question in the context of limit complexity in the categories of semi-linear and semi-algebraic sets, and answer it in the negative.
Our final result is an analysis of the colimit complexity of the image functor for the category of modules over polynomial rings.

Finally, in Section~\ref{sec:open} we list several open problems and future research direction in the area of categorical complexity.

We assume no prior knowledge of category theory in this paper and have included all relevant definitions. For background in category theory
we refer the reader to the books \cite{Awo, Maclane, Simmons}.
For background in complexity theory  we refer the reader to the books \cite{Burgisser-book1,Burgisser-book2}.
Finally, we make use of certain basic functors from algebraic geometry, and we refer the reader to the book \cite{Mumford-Oda} as an accessible source for these.

\section{Categories and functors}
\label{sec:definitions}

In this section we recall some basic definitions from category theory, and introduce some notation that will be useful in what follows.

\begin{definition}[Categories]
\label{def:categories}
A category $\cC$ consists of:
\begin{enumerate}[1.]
\item
a class $\Ob(\cC)$ whose elements are the ``objects of the category $\cC$'';
\item
for every pair $A,B$ of objects of $\cC$, a set $\cC(A,B)$ of ``morphisms'' or ``arrows'' from $A$ to $B$; 
\item
for every triple $A,B,C$ of objects of $\cC$, a composition law
\[
\cC(A,B) \times \cC(B,C) \rightarrow \cC(A,C)
\]
which will be denoted by $(f,g) \mapsto g\circ f$;
\item
for every object $A$ of $\cC$, a morphism $1_A \in \cC(A,A)$ called the identity morphism on $A$.
\end{enumerate}
The above data are subject to the following two axioms.
\begin{enumerate}[(a)]
\item
\label{itemlabel:def:categories:associativity}
(Associativity). Given morphisms $f \in \cC(A,B), g \in \cC(B,C), h \in \cC(C,D)$,
\[
h \circ (g\circ f) = (h \circ g)\circ f.
\]
\item
\label{itemlabel:def:categories:identity}
(Identity). Given morphisms $f \in \cC(A,B), g \in \cC(B,C)$ the following equalities hold.
\begin{eqnarray*}
1_B \circ f &=& f, \\
g \circ 1_B &=& g.
\end{eqnarray*}
\end{enumerate}
We say that a category $\cC$ is a small category if its class of objects is a set.

For any category $\cC$, we will denote $\cC^{\mathrm{opp}}$ the category, whose morphisms are defined by 
\[
\cC^{\mathrm{opp}}(A,B) = \cC(B,A),
\]
for every pair of objects $A,B$ of $\cC$.
\end{definition}

\begin{notation}
We will denote a morphism $f \in \cC(A,B)$ often as $f:A \to B$, and also denote the source $A$ by $\op{dom}(f)$, and the target
$B$ by $\op{codom}(f)$. 
\end{notation}

\begin{notation}
\label{not:categories}
The following categories will appear later in the paper.
\begin{enumerate}
\item
\label{itemlabel:not:categories:set}
The category $\opcat{Set}$ whose objects are sets and whose morphisms are maps between sets.
\item
\label{itemlabel:not:categories:Vect}
The category $\opcat{Vect}_k$ where $k$ is a field, and whose objects are $k$-vector spaces, and whose morphisms are linear maps.
\item
\label{itemlabel:not:categories:Grp}
The category $\opcat{Grp}$ of groups and homomorphisms.
\item
\label{itemlabel:not:categories:SL-SA}
The category $\opcat{SL}$ (respectively, $\opcat{SA}$)  of embedded semi-linear (respectively, semi-algebraic) sets and 
affine (respectively, polynomial) maps. More precisely,
each object of $\opcat{SL}$ (respectively, $\opcat{SA}$)is  a semi-linear 
(respectively, semi-algebraic) subset  $A \subset \R^n$ for some $n \geq 0$, and a morphism
$(A \subset \R^n) \rightarrow (B \subset \R^m)$ is a map $f: A \rightarrow B$,  such that there exists a commutative square
\[
\xymatrix{
A \ar@{^{(}->}[d] \ar[r]^{f} & B \ar@{^{(}->}[d]\\
\R^n \ar[r]^g & \R^m
}
\]
where $g:\R^n \rightarrow \R^m$ is an affine (respectively, polynomial) map. In other words $f$ is the restriction to $A$
of an affine (respectively, polynomial) map from $\R^n$ to $\R^m$.
\item
\label{itemlabel:not:categories:Alg}
The category $\opcat{Alg}_k$ of $k$-algebras over a field $k$.
\item
\label{itemlabel:not:categories:AffVar}
The category $\opcat{AffVar}_k$ of affine varieties, and  the category $\opcat{AffSch}_k$ of affine $k$-schemes for a field $k$.
\item
\label{itemlabel:not:categories:top}
The category of $R$-$\opcat{Mod}$  where $R$ is a polynomial ring in finitely many variables.
\item
\label{itemlabel:not:categories:set}
The category of $\Top$ of topological spaces.
\end{enumerate} 
\end{notation}

\begin{definition}[Functors]
\label{def:functor}
A (covariant)  functor $F$ from a category $\cA$ to $\cB$ consists of the following:
\begin{enumerate}[1.]
\item
a mapping $\Ob(\cA) \rightarrow \Ob(\cB)$  (the image of $A$ will be written as $F(A)$);
\item
for every pair objects $A,A'$ of $\cA$, a mapping $\cA(A,A') \rightarrow \cB(F(A), F(A'))$ (the image of $f \in \cA(A,A')$) is written as $F(f)$).
\end{enumerate}
The above data is subject to the following axioms.
\begin{enumerate}[(a)]
\item
\label{itemlabel:def:functor:composition}
for every pair of morphisms $f\in \cA(A,A'), g \in \cA(A',A'')$, 
\[
F(g \circ f) = F(g) \circ F(f);
\]
\item
\label{itemlabel:def:functor:identity}
for every object $A$ of $\cA$,
\[
F(1_A) = 1_{F(A)}.
\]
\end{enumerate}

A contravariant functor $F$ from a category $\cA$ to $\cB$ is a covariant functor from
$\cA^{\mathrm{opp}}$ to $\cB$.
\end{definition}

\begin{definition}[Natural transformations between functors]
\label{def:nat}
Let $F,G$ be two (covariant) functors from a category $\cC$ to $\cD$. A natural transformation $\theta:F \rightarrow G$, is a family of arrows 
$
\left(\theta_C: F(C) \rightarrow G(C)\right)_C$ in $\cD$, indexed by objects $C$ of $\cC$, 
such that for each pair of objects $C,C'$ of $\cC$, and $f \in \cC(C,C')$ the following diagram commute:
\[
\xymatrix{
F(C) \ar[d]^{F(f)} \ar[r]^{\theta_C} & G(C) \ar[d]^{G(f)}\\
F(C')    \ar[r]^{\theta_{C'}} & G(C).
}
\]

We will denote the class of all natural transformations between two functors $F,G$ by 
$\Nat(F,G)$. A natural transformation $\theta \in \Nat(F,G)$ is called an isomorphism
if it admits an inverse.

\end{definition}

\begin{definition}
\label{def:adjoint}
Two functors $F:\cC \rightarrow \cD$, and $U:\cD \rightarrow \cC$ are said to be an \emph{adjoint pair} (with $F$ left adjoint to $U$, and $U$ right adjoint to $F$), if there exists for each object $C$ of $\cC$ and $D$ of $\cD$, bijective maps
\[
\cC(C,U(D)) \rightarrow \cD(F(C),D), 
f \mapsto f^\sharp,
\]
\[
\cD(F(C),D) \rightarrow \cC(C,U(D)), g \mapsto g_\flat
\]
which are  inverses to each other and are moreover natural in $C$ and $D$.

More precisely, 
for every $k \in \cC(A,C), \ell \in \cD(D,B)$, 
\[ 
(U(\ell)\circ f \circ k)^\sharp = \ell \circ f^\sharp \circ F(k),
\]
\[
U(\ell) \circ g_\flat \circ k = (\ell \circ g \circ F(k))_\flat
\]
both hold.
\end{definition}

\begin{remark}
\label{rem:adjoint}
In Definition~\ref{def:adjoint},
the functors $F,U$ induce functors $\mathcal{F},\mathcal{U}$  from the product  category
$\cC^{\mathrm{opp}} \times \cD$ to $\opcat{Set}$ defined by
\[
\mathcal{F}(C,D) = \cD(F(C),D), 
\]
\[
\mathcal{U}(C,D) = \cC(C, U(D)),
\]
with $\mathcal{F}(f,g), \mathcal{G}(f,g)$ defined in the obvious manner for 
$f,g$ arrows in the categories $\cC^{\mathrm{opp}}$ and $\cD$ respectively.
Then, the naturality condition in Definition~\ref{def:adjoint}  translates into the fact that the functors $\mathcal{F}$
and $\mathcal{U}$ are natural transformations.
\end{remark}

\begin{example}
\label{eg:adjoint}
The functors $F: \opcat{Set} \rightarrow \opcat{Grp}$ which takes a set to the free group generated by the set, and the forgetful functor $U: \opcat{Grp} \rightarrow \opcat{Set}$ form an adjoint pair (with $F$ left adjoint to $U$).
\end{example}

We also need the notion of universal elements.  

\begin{definition}[Representable functors, Yoneda's Lemma, and universal elements]
\label{def:universal}
Let $\cC$ be a category and $F$ a functor from $\cC$ to $\opcat{Set}$. Then for each
object $A$ of $\cC$, $\cC(A,\cdot)$ is a functor from $\cC$ to $\opcat{Set}$. The map
\[
\lambda_A: \Nat(\cC(A,\cdot),F) \rightarrow F(A),
\]
defined by
\[
\lambda_A(\phi) = \phi_A(\id_A)
\]
is bijective. This statement is referred to as Yoneda's Lemma.
Now, if $\phi \in  \Nat(\cC(A,\cdot),F)$ is an isomorphism, then we say that $F$ is representable (by $A$), and $u = \lambda_A(\phi)$ is called a \emph{universal element} of $F$. The element $u$ has the property, that for any object $B$ of $\cC$ and $f \in \cC(B,A)$, there is a unique element
 $t \in F(B)$, such that $u = F(f)(t)$.
 
If $F$ is a contravariant functor from $\cC$ to $\opcat(Set)$, then it is representable by an
object $A$ of $\cC$, if and only if the corresponding covariant functor 
from $\cC^{\mathrm{opp}}$ to sets is representable, and in this case a universal element of 
the this covariant functor will be called a universal element of $F$.
\end{definition}

\begin{definition}[Graphs and diagram categories]
\label{def:diagram}
A  \emph{directed graph} $I$ is a quadrapule  $(V,E,s,t)$, where $V,E$ are sets (referred to as the sets of vertices and edges  of $I$), and $s,t: E \to V$ are maps. (For $e \in E$,
we will sometime refer to $s(e)$ as the source, and $t(e)$ as the target of the edge $e$.) A homomorphism $\phi = (\phi_1,\phi_2):I \rightarrow I'$  of directed graphs, 
$I = (V,E,s,t), I' = (V',E',s',t')$ is a pair of maps $\phi_1:V \rightarrow V', \phi_2:E \rightarrow E'$, such that the two diagrams
\[
\xymatrix{
E \ar[r]^{\phi_2} \ar[d]^{s} &E'\ar[d]^{s'} \\
V\ar[r]^{\phi_1} & V'
},
\xymatrix{
E \ar[r]^{\phi_2} \ar[d]^{t} &E'\ar[d]^{t'} \\
V\ar[r]^{\phi_1} & V'
} 
\]
commute.

For a small category $\cC$, we will denote by $U(\cC)$ the directed graph, whose set of vertices is the set of objects of $\cC$, and  whose set of edges is the set of all morphisms of
$\cC$, along with the maps $s,t$ taking a morphism to its codomain and  domain, respectively. Even if the category $\cC$ is not small we will continue to use the notation $U(\cC)$ to denote its underlying graph, keeping in mind that will always restrict our attention to subgraphs of $U(\cC)$ with finite sets of vertices.

Let $\cC$ be a category and let $U(\cC)$ be the underlying directed graph. Let $I=(V,E,s,t)$ be any directed graph. By a \emph{diagram in $\cC$}, we mean a directed graph homomorphism $D: I \to U(\cC)$. The graph $I$ will be  called the \emph{shape of $D$}, 
and 
we will denote by $\op{v}(D)$ (respectively, $\op{e}(D), s(D), t(D)$), the set of vertices (respectively, edges, sources and targets of edges) of the graph $I$.

We say that the diagram
$D$ is \emph{discrete}, if the edge set $E$ of $I$ is empty.

By a \emph{subdiagram} of a diagram $D: I\to U(\cC)$, with $I=(V,E)$, we mean the restriction $D_{J}: J \to U(\cC)$, with $J=(V',E')$ a full sub-graph of $I$, i.e. $V'\subset V$, and $E' = \left\{ e\in E \,|\, s(e)\in V'\t{, }t(e)\in V' \right\}$. The restrictions to not necessarily full subgraphs will be specified as \emph{not necessarily full subdiagrams}.

For $I=(V,E,s,t)$, and two diagrams, $D_1 : I \to U(\cC)$, $D_2: I \to U(\cC)$, a morphism between $D_1$ and $D_2$ is a collection of morphisms $\varphi = (\varphi_{v} : D_1(v) \to D_2(v))_{v\in V}$, such that, for all $e\in E$, $D_2(e) \compose \varphi_{D_1(s(e))} = \varphi_{D_1(t(e))} \compose D_1(e)$. This defines the \emph{category of diagrams, $\cC^{I}$,  of the category $\cC$ with shape $I$}.  
\end{definition}

\begin{remark}
\label{rem:path-cat}
Note that in Definition~\ref{def:diagram} $I$ is just a directed graph and does not have a composition operation on it. As such, there is no a priori assumption of functoriality/commutativity for diagrams. 

It is possible to associate a category, $\opcat{Pth}(I)$,  to each directed graph $I$, called the path category of $I$,
and define (path) diagrams as actual functors  from $\opcat{Pth}(I)$ to a category $\cC$ \cite{Simmons}. The diagrams we consider in this paper are not functors in this sense.

Also, note that our notion of a full subdiagram is not the same
as the full sub-functor of the diagram functor from the path category of $I$. So if $I$ is the directed graph $1 \rightarrow 2 \rightarrow 3$, and $D$ a diagram of a category $\cC$ of shape $I$, then the restriction
of $D$ to the full subgraph corresponding to the vertex set $J = \{1,3\}$, 
is the diagram whose image is the subgraph of $U(\cC)$ consisting of two vertices 
$D(1),D(3)$, and an empty set of edges.  
\end{remark}

\begin{definition}[Cones, limits]
\label{def:limit}
Given an object $W$ of $\cC$, a \emph{constant diagram with value $W$ and shape $I=(V,E,s,t)$}, 
is the diagram $D: I \to U(\cC)$, with $D(v)  = W$ for all
$v \in V$, and $D(e) = 1_W$ for all $e \in E$. 

\begin{figure}
\includegraphics[scale=0.7]{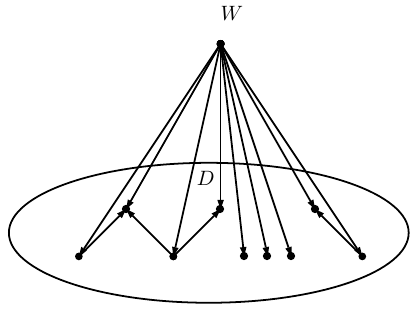}
\caption{Cone over $D$ with vertex $W$.}
\label{fig:cone}
\end{figure}

Given a diagram $D: I \to U(\cC)$ and an object $W$ of $\cC$, a \emph{cone over $D$ with vertex $W$},  is a morphism $D' \to D$, 
where $D'$ is a constant diagram of shape $I$ with value $W$. 
Notice that if $D'$ is the constant diagram of shape $I$ and value $W$, then a cone $\varphi:D' \to D$ is determined by morphisms 
$\varphi(v): W \rightarrow D(v)$, such that following diagram commutes  
for each $e \in E$:

\begin{equation}
\label{eqn:cone-commutative}
\xymatrix{
& W \ar[ld]_{\varphi(s(e)}\ar[rd]^{\varphi(t(e))} & \\
D(s(e)) \ar[rr]^{D(e)} && D(t(e)).
}
\end{equation}

Note that the cone over $D:I\to \cC$ with vertex $W$ can also be thought of as a new diagram, whose associated directed graph is obtained from the graph of $I$ by adding has one extra vertex $v_0$, and an edge, $e_v$,  for each vertex $v$ of $I$, with source $v_0$ and target $v$
(cf. Figure \ref{fig:cone}). 
This diagram of course has the extra property that all subdiagrams of the form
\eqref{eqn:cone-commutative}  commute.

Given a diagram $D:I \rightarrow U(\cC)$, with $I = (V,E,s,t)$, 
the map $\cC^{I}(\cdot,D)$, that associates to each object $W$ of $\cC$ the set of cones over $D$ with vertex $W$ defines a contravariant functor from $\cC$ to $\opcat{Set}$. 

A universal element of this functor is called the \emph{limit of $D$} (denoted by $\lim D$). In other words, $\lim D$ is a cone $\lim D: D' \rightarrow D$, where $D'$ is a constant diagram of shape $I$ (say with value $L$), such that for any cone
$\varphi: D'' \rightarrow D$ where $D''$ is a constant diagram of shape $I$ with value $M$, there exists a unique morphism $\varphi':M \rightarrow L$ such that the following diagram commutes (denoting by $\varphi'$ also the induced morphism in $\cC^I(D'',D)$):

\[
\xymatrix{
D'' \ar[rr]^{\varphi'}\ar[rd]^{\varphi}&& D'\ar[ld]_{\lim D}\\
&D&
}.
\]

We will denote by $I(\lim D)$, the directed graph $(V',E',s',t')$ 
where 
\begin{eqnarray*}
V' &=& V \cup \{v_0\}, v_0 \not \in V,\\
E' &=& E \cup \bigcup_{v \in V} \{e_v\},
\end{eqnarray*}
and $s',t'$ are defined by
\[
\left.
\begin{array}{lll}
s'(e) &=& s(e),\\
t'(e) &=& t(e),
\end{array}
\right\} 
\mbox{ if $e \in E$}, 
\]
\[
\left.
\begin{array}{lll}
s'(e) &=& v_0,\\
t'(e) &=& v.
\end{array}
\right\}
\mbox{ if $e = e_v, v \in V$}.
\]
We will also sometime abuse notation and refer to the value $L$ of the limit also by $\lim D$.
\end{definition}

\begin{example}[Products, pullbacks, equalizers, composition]
\label{eg:examples-of-limits}
The limit of a discrete diagram is called a \emph{product}. Limits of diagrams of shapes,
\[
\xymatrix{
\bullet \ar[r] &\bullet & \bullet \ar[l]}, 
\xymatrix{
 \bullet \ar@<-.5ex>[r] \ar@<.5ex>[r] & \bullet,
}
\]
are usually referred to as pull-backs and equalizers respectively.
Note also that taking the  composition of two morphisms is a particular case of taking limits of diagrams of shape
\[
\xymatrix{
\bullet \ar[r]& \bullet \ar[r]& \bullet
}.
\]
\end{example}

\begin{definition}[Cocones, colimits]
\label{def:colimit}
By dualizing (i.e. reversing the direction of the arrow) in Definition \ref{def:limit} we obtain the notion of cocones and colimits of diagrams.
Given a diagram $D: I \to U(\cC)$ and an object $W$ of $\cC$, a \emph{cocone over $D$ with vertex $W$},
is a morphism $D \to D'$, 
where $D'$ is a constant diagram of shape $I$ with value $W$. 
Notice that if $D'$ is the constant diagram of shape $I$ and value $W$, then a cone $\varphi:D \to D'$ is determined by morphisms 
$\varphi(v): D(v) \rightarrow W$, such that the following diagram commutes for each $e \in E$:

\[
\xymatrix{
& W  & \\
D(s(e)) \ar[ru]^{\varphi(s(e)}\ar[rr]^{D(e)} && D(t(e))\ar[lu]_{\varphi(t(e))}
}.
\]

Given a diagram $D:I \rightarrow U(\cC)$, with $I = (V,E,s,t)$, 
the map $\cC^{I}(D,\cdot)$, that associates to each object $W$ of $\cC$ the set of cocones over $D$ with vertex $W$ defines a covariant functor from $\cC$ to $\opcat{Set}$. A universal element of this functor is called the \emph{colimit of $D$} (denoted by $\colim D$). In other words, $\colim D$ is a cocone $\colim D: D \rightarrow D'$, where $D'$ is a constant diagram of shape $I$ (say with value $C$), such that for any cocone
$\varphi: D \rightarrow D''$ where $D''$ is a constant diagram of shape $I$ with value $M$, there exists a unique morphism $\varphi':C \rightarrow M$ such that the following diagram commutes (denoting by $\varphi'$ also the induced morphism in $\cC^I(D'',D)$):

\[
\xymatrix{
D' \ar[rr]^{\varphi'} && D''\\
&D\ar[lu]^{\colim D}\ar[ru]^{\varphi}&
}.
\]

We will denote by $I(\colim D)$, the directed graph $(V',E',s',t')$ 
where 
\begin{eqnarray*}
V' &=& V \cup \{v_0\}, v_0 \not \in V,\\
E' &=& E \cup \bigcup_{v \in V} \{e_v\},
\end{eqnarray*}
and $s',t'$ are defined by
\[
\left.
\begin{array}{lll}
s'(e) &=& s(e),\\
t'(e) &=& t(e),
\end{array}
\right\} 
\mbox{ if $e \in E$}, 
\]
\[
\left.
\begin{array}{lll}
s'(e) &=& v,\\
t'(e) &=& v_0.
\end{array}
\right\}
\mbox{ if $e = e_v, v \in V$}.
\]
We will also sometime abuse notation and refer to the value $C$ of the colimit also by $\colim D$.
\end{definition}

\begin{example}[Coproducts, push-forwards and coequalizers]
\label{eg:examples-of-colimits}
By reversing the direction of the arrows of the diagrams in Example \ref{eg:examples-of-limits}, and taking colimits, we obtain
the definitions of coproducts, push-forwards and coequalizers (respectively, from the definitions of products, pullbacks and equalizers).
\end{example}


\section{Definition of diagram computations}
\label{sec:definitions}

In this section we define \emph{ diagram computations} and \emph{categorical complexity}. As explained in 
Section~\ref{sec:intro} (Introduction)
diagram computations come in three
different flavors -- namely limit, colimit and mixed limit-colimit computations. The associated notions of complexities of diagrams will be called
limit, colimit and mixed complexity respectively.   

We fix a category $\cC$ for the rest of this section, and also fix a set $\cA$ of morphisms in $\cC$. The morphisms in $\cA$  will be referred to as the \emph{basic morphisms} in $\cC$. 

\subsection{Limit and colimit computations}
\label{subsec:limit-computation}
We define a notion of computation in $\cC$, called a $\emph{limit computation}$ by starting with these basic morphisms and adding a finite limit at each step; similarly, in a \emph{colimit computation}, we build objects by adding colimits of subdiagrams. 

\begin{definition}
\label{def_categorical-computation}
  A \emph{limit computation} (respectively, a \emph{colimit computation}) in $\cC$ is a finite sequence of diagrams $(D_0,\dots,D_s)$, with $D_j: I_j 
  = (V_j,E_j,s_j,t_j)
  \to U(\cC), 0 \leq j \leq s$, where: 
 \begin{enumerate}[(i)]
   \item
   \label{itemlabel:def:categorical-computation:1}
   $D_0(e) \in \cA$ for each edge $e$ of $I_0$.
   \item
   \label{itemlabel:def:categorical-computation:2}
     For each $i=1,\dots,s$, $D_{i}$ is obtained from $D_{i-1}$ by adding a limit or colimit cone of a subdiagram. More precisely, there is a subdiagram $D_{i-1}|_{J_i}$ of $D_{i-1}$ with 
     \[
     J_i = (V_{i-1}',E_{i-1}',s_i' = s_i|_{E_{i-1}'},t_i' = t_i|_{E_{i-1}'}), V_{i-1}' \subset V_{i-1}, E_{i-1}' \subset E_{i-1},
     \]
     \hide{
     and an object $L_i$ which is a limit (respectively, $C_i$ which is a colimit) of $D_{i-1}|_{J_i}$ 
     such that the difference between $D_{i}$ and $D_{i-1}$ are $L_i$ and the limit cone morphisms out of $L_i$ (respectively, $C_i$ and the colimit cocone morphisms into $C_i$). 
    }
    such that, denoting $I(\lim D_{i-1}|_{J_i}) = (V',E',s',t')$ (respectively, $I(\colim D_{i-1}|_{J_i}) = (V',E',s',t')$),
    \begin{enumerate}
    \item
    $V_i = V_{i-1} \cup \{v_0\}$, where $v_0$ is the unique vertex of $V' \setminus V_{i-1}$;
    \item 
    $E_i = E_{i-1} \cup \bigcup_{v \in V_{i-1}'} \{e_v\}$;
    \item
    For $e \in E_{i}$, 
    \[
   \left.
    \begin{array}{lll}
    s_i(e) &=& s_{i-1}(e),\\
    t_i(e)   &=& t_{i-1}(e),
    \end{array}
   \right\} 
  \mbox{ for } e \in E_{i-1},
  \]
  \[
   \left.
    \begin{array}{lll}
 s_i(e) &=& v_0 \mbox{ (respectively, $v$)}, \\
 t_i(e) &=& v \mbox{ (respectively, $v_0$)},
    \end{array}
   \right\} 
 \mbox{ for } e= e_v,  v \in V_{i-1}';
  \]
  \item
    For $e \in E_{i}$, 
    \begin{eqnarray*}
    D_i(e) &=& D_{i-1}(e) \mbox{ for } e \in E_{i-1}, \\
    D_i(e)  &=& (\lim D_{i-1}|_{J_i})_v \mbox {(respectively,  $(\colim D_{i-1}|_{J_i})_v$), for $e = e_v,  v \in V_{i-1}'$},
    \end{eqnarray*}
 (cf. Definitions \ref{def:diagram}, \ref{def:limit} and  \ref{def:colimit}).
    \end{enumerate}
   \item
   \label{itemlabel:def:categorical-computation:3} (\emph{Constructivity}) 
   For each  $i, 0 < i \leq s$,
   if the unique vertex $v_0 \in V_i \setminus V_{i-1}$ belongs to $J_j$ for some $j, i < j \leq s$, 
   then $J_i$ is a subgraph of $J_j$. (In other words,
   if a limit $L_i = \lim D_{i-1}|_{J_i}$ (respectively, colimit $C_i$) produced in the $i$th step of the computation is used again in the    
   subdiagram $D_{j-1}|_{J_j}$ used at the $j$th step of the computation, then 
   $J_{i}$ is a subgraph of $J_{j}$, i.e. the subdiagram that produced $L_i$ (respectively, $C_i$) must be a subdiagram of 
   $D_{j-1}|_{J_j}$.)
 \end{enumerate}

The computation $(D_0,\dots,D_s)$ is said to \emph{compute a diagram $D$}, if $D$ is isomorphic to a (not necessarily full) subdiagram of $D_s$. In particular, an object in $\cC$ is computed by $(D_0,\dots,D_s)$ if an object isomorphic to it appears in $D_s$. 
\end{definition}

\begin{remark}
\label{rem:categorical-computation}
Note that  in Definition~\ref{def_categorical-computation} we are not assuming
that all limits or colimits of finite diagrams exist but in each  particular computation 
$(D_0,\dots,D_s)$, for each $i, 0 \leq i < s$, the limit or colimit of the subdiagram
$D_i|_{J_i}$ of $D_i$
is assumed to exist. 
Thus, the notion of a limit/colimit
computation is still well defined even if the category $\cC$ does not admit limits or colimits of all finite diagrams. See also Remark~\ref{rem:uncomputable} below.
\end{remark}

 \begin{remark}
 \label{rem:constructive}
 One could take Parts \eqref{itemlabel:def:categorical-computation:1} and \eqref{itemlabel:def:categorical-computation:2}
 as the definition of limit (respectively, colimit) computation. However, in order that our notion of categorical complexity is closer to the classical notions -- such as circuit complexity  in certain  relevant categories (see Section \ref{sec:circuits}), 
we also consider the constructivity condition; which roughly means that the limit (or colimit) computation does not forget how an object was constructed. It also prevents objects of exponential rank/size from being constructed; cf.~Example~\ref{ex_nonconstructiveexample}.
 \end{remark}

\begin{remark}
\label{rem:uncomputable}
Of course, one may not be able to obtain every object, morphism or diagram from a given set of basic morphisms $\cA$ in a category $\cC$. We will think of such objects/morphisms/diagrams as non-computable in $\cC$ with respect to $\cA$. 
\end{remark}

We now describe a basic syntax for writing down the limit or colimit computations. The computation is described by a set expressions, each expression in a line. The first kind of expression is of the form  
$$\code{i. source}\code{,}f\code{,target}$$
and describes objects and/or morphisms that are added to $D_0$. Here, $\code{i}$ is an identifier that can be any string. In subsequent lines, the identifier `\code{i}' is used to refer to the source, and `\code{i'}' is used to refer to the target of the basic morphism $f \in \cA$ that is added to $D_0$ by this expression. \code{source} and \code{target} are the identifiers of the vertices which are the intended source and target of the new morphism being attached to $D_0$. 
 
If the source is a new vertex that did not exist in the diagram before, then we write \code{i.\,i}\code{,}$f$\code{,target}, or \code{i.\,\_}\code{,}$f$\code{,target} for it. If only the target is new, we write \code{i.\,source}\code{,}$f$\code{,i'} or \code{i.\,source}\code{,}$f$\code{,\_}; we write \code{i.\,i}\code{,}$f$\code{,i'} or \code{i.\,\_}\code{,}$f$\code{,\_} if both are new, distinct vertices,
and \code{i.\,\_}\code{,}$f$\code{,i} if both are new and are the same vertex.

There is no need to list all the morphisms in $D_0$ at the beginning, so we will have these steps as intermediate steps as well; as long as the morphisms attach only to other vertices in $D_0$, they be can be considered as part of $D_0$. 

The second kind of expression are those of the form:
$$\code{i. lim(a,b,\dots)}$$
which describe steps where a limit is added to the subdiagram. The identifiers $\code{a,b,\dots}$ describe the vertices in the 
subdiagram whose limit is being taken. In subsequent steps, \code{i} is used to refer to the limit added. Similarly, we write 
$\code{i. colim(a,b,\dots)}$
for describing colimit computations. We may use the notation \code{i->a} to refer to morphisms created during the computation.   

We start with two basic examples about constructions in the category of sets. 

\begin{example}\label{ex_set} 
Let $\cC$ be the category of sets and let $\cA$ consist of a single morphism $\op{id}: \left\{ 1 \right\} \to \left\{ 1 \right\}$. 
Consider the colimit computation described by
  \begin{tabularx}{0.5\textwidth}{ll}
    \code{1.} & \code{\_,}$\left\{ 1 \right\} \xtonormal{\id} \left\{1
    \right\}$\code{,1} \\
    \code{2.} & \code{\_,}$\left\{ 1 \right\} \xtonormal{\id} \left\{1
    \right\}$\code{,2} \\
    \dots & \\
    \code{n.} & \code{\_,}$\left\{ 1 \right\} \xtonormal{\id} \left\{1
    \right\}$\code{,n} \\
  \code{n+1.} & \code{colim(1,2,\dots,n)} 
\end{tabularx}
For each $k\leq n$, the step \code{k.} \code{\_,}$\left\{ 1 \right\} \xto{\id} \left\{1 \right\}$\code{,k} is adding a new copy of $\left\{ 1 \right\}$ to the diagram  (i.e. the vertex labelled $\texttt{k}$ with a self-loop corresponding to $\mathrm{id}$). In the end, \code{n+1} is the set with $n$ elements. 
\end{example}

\begin{example}\label{ex_setmorphism}
  Continuing with the previous example, we now make a colimit computation that produces the morphism $\left\{ 0,1,2 \right\} \xto{f} \left\{ 0,1,2 \right\}$ in the category of sets, where $f(0) = 0$, $f(1) = 0$, $f(2) = 1$.

\begin{minipage}[c]{0.5\textwidth}
\includegraphics[width=2.6in]{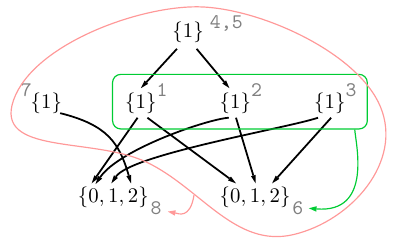}
\end{minipage}
\begin{minipage}[c]{0.5\textwidth}
\begin{tabularx}{0.5\textwidth}{ll}
    \code{1.} & \code{\_,}$\left\{ 1 \right\} \xtonormal{\id} \left\{1
    \right\}$\code{,1} \\
    \code{2.} & \code{\_,}$\left\{ 1 \right\} \xtonormal{\id} \left\{1
    \right\}$\code{,2} \\
    \code{3.} & \code{\_,}$\left\{ 1 \right\} \xtonormal{\id} \left\{1
    \right\}$\code{,3} \\
    \code{4.} & \code{\_,}$\left\{ 1 \right\} \xtonormal{\id} \left\{1
    \right\}$\code{,1} \\
    \code{5.} & \code{4,}$\left\{ 1 \right\} \xtonormal{\id} \left\{1
    \right\}$\code{,2} \\
    \code{6.} & \code{colim(1,2,3)}\\ 
    \code{7.} & \code{\_,}$\left\{ 1 \right\} \xtonormal{\id} \left\{1
    \right\}$\code{,7} \\
    \code{8.} & \code{colim(1,2,3,4,6,7)}\\ 
\end{tabularx}
\end{minipage}
\vspace{0.1in}

The morphism \code{6->8} is $f$, in the sense that the full-subdiagram containing \code{6} and \code{8} is isomorphic to $\left\{ 0,1,2 \right\} \xto{f} \left\{ 0,1,2 \right\}$.  
\end{example}

We will come back to sets later. We now discuss a more detailed example where we annotated each step in the computation. 

\begin{example}\label{ex_kvect-1}
  Consider the category $\opcat{Vect}_k$ of vector spaces over a field $k$. Let $\cA$ consist of the scalar multiplication morphisms $k \xto{c} k$ for each $c\in k$, the addition morphism $k^2 \xto{+} k$, 
  the two projections $\pi_1,\pi_2:k^{2} \to k$, and morphisms $0 \to k$, $k\to 0$. Say, the characteristic of $k$ is $0$ and we wish to compute the morphism 
$f: k^3 \to k^2$, $f(x,y,z) = (2x+2y+3z,y+z)$. 
  We describe the computation as follows.

\begin{minipage}[c]{0.5\textwidth}
\includegraphics[width=2.4in]{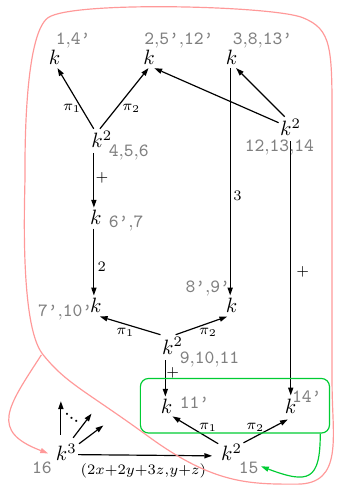}
\end{minipage}
\begin{minipage}[c]{0.5\textwidth}
\begin{tabularx}{0.5\textwidth}{ll}
  \code{1.} \code{\_,}$k \xtonormal{1} k$\code{,1}   \\ 
  \code{2.} \code{\_,}$k \xtonormal{1} k$\code{,2} \\ 
  \code{3.} \code{\_,}$k \xtonormal{1} k$\code{,3} \\
  \code{4.} \code{\_,}$k\times k\xtonormal{\pi_1} k$,\code{1} \\ 
  \code{5.} \code{4,}$k \times k \xtonormal{\pi_2} k$,\code{2} \\ 
  \code{6.} \code{4,}$k\times k \xtonormal{+} k$\code{,6'} \\
  \code{7.} \code{6',}$k \xtonormal{2} k$\code{,7'} \\
  \code{8.} \code{3,}$k \xtonormal{3} k$\code{,8'} \\
  \code{9.} \code{\_,}$k\times k\xtonormal{\pi_2} k$,\code{8'} \\ 
  \code{10.} \code{9,}$k\times k\xtonormal{\pi_1} k$,\code{7'} \\ 
  \code{11.} \code{9,}$k\times k \xtonormal{+} k$\code{,11'} \\
  \code{12.} \code{\_,}$k\times k\xtonormal{\pi_1} k$,\code{2} \\ 
  \code{13.} \code{12,}$k \times k \xtonormal{\pi_2} k$,\code{3} \\ 
  \code{14.} \code{12,}$k\times k \xtonormal{+} k$\code{,14'} \\
  \code{15.} \code{lim(11',14')} \\
  \code{16.} \code{lim(1,1',2,2',\dots} \\ 
  $\,\,\,\,\,\,\,\,\,\,\,\,\,\,$ \code{\dots,14,14',15)} \\
\end{tabularx}
\end{minipage}
\end{example}

\begin{remark}\label{rem_compositionsanduniversal}
Two facts about limits and colimits are useful in thinking about the above example and other computations. 

The first  (as already noted in Example \ref{eg:examples-of-limits}) is that  if $f: X\to Y$ and $g:Y\to Z$ are morphisms, then the limit of the diagram $X \xto{f} Y \xto{g} Z$ is (isomorphic to) $X$, 
 and the induced morphism $X \rightarrow Z$ is equal to the composition $g \circ f$. So, compositions are obtained using limits.
 \hide{
The second is that if we take the limit $L$ of a diagram $D:I = (V,E,s,t) \to \cC$, and $X$ is a cone over $D$, 
then,  the limit  of these two diagrams joined together produces an object isomorphic to $X$. The map $X\to L$ is then the map that would normally come from the universal property of the limit $L$. 
}
The second fact is the following. Let $D:I = (V,E,s,t)\to \cC$ be a diagram, 
and $L = \lim D$, and let 
$\varphi:C \to D$ a cone over $D$, where $C$ is a constant diagram of shape $I$ and value $X$. 
Let $I'=(V',E,s',t')$ be the graph, with $V' = V \cup \{v_0,v_1\}$, $E' = E \cup \bigcup_{v \in V} \{e_v,e'_v\}$,
\[
\left.
\begin{array}{lll}
s'(e) &=& s(e), \\
t'(e)  &=& t(e),
\end{array}
\right\} \mbox{ for $e \in E$,}
\]
$$\displaylines{
\left.
 \left.
\begin{array}{lll}
s'(e) &=& v_0, \\
t'(e)  &=& v,
\end{array}
\right\} \mbox{ for $e = e_v$,}  
\left.
\begin{array}{lll}
s'(e) &=& v_1, \\
t'(e)  &=& v,
\end{array}
\right\} \mbox{ for $e = e_v'$,}
\right\} \mbox{ for $v \in V$},
}
$$
and $D':I' \to \cC$ the diagram defined by,
\[
\begin{array}{lll}
D'(v) &=& D(v), \mbox{ for $v\in V$,}\\
D'(v_0)&=& L,\\
D'(v_1) &=& X,
\end{array}
\]
\[
\begin{array}{lll}
D'(e) &=& D(e), \mbox{ for $e \in E$,}\\
D'(e) &=& (\lim D)_v, \mbox{ for $e = e_v, v \in V$,}\\
D'(e) &=& \varphi_v, \mbox{ for $e = e_v', v \in V$.}
\end{array}
\]
Then, $\lim D'$ is isomorphic to $X$, and morphism $X\to L$ in this corresponding limit cone, is the unique morphism coming from the  universal property of the limit $L$. 
Thus, in a limit computation, once the cones corresponding to $L$ and $X$ are computed,
in order to obtain the morphism $X \rightarrow L$, that is implied by the universal property of limits, 
one needs to take just one additional limit.  

A similar fact holds for colimits (with the arrows reversed). In other words, 
let $D:I = (V,E,s,t)\to \cC$ be a diagram, 
and $M = \colim D$, and let 
$\varphi:D \to C$ a cocone over $D$, where $C$ is a constant diagram of shape $I$ and value $Y$. 
Then, in a colimit computation, once the cocones corresponding to $M$ and $Y$ are computed,
in order to obtain the morphism $M \rightarrow Y$, that is implied by the universal property of colimits, 
one needs to take just one additional colimit.
\begin{figure}
\hspace{-1.5in}\includegraphics[scale=0.7]{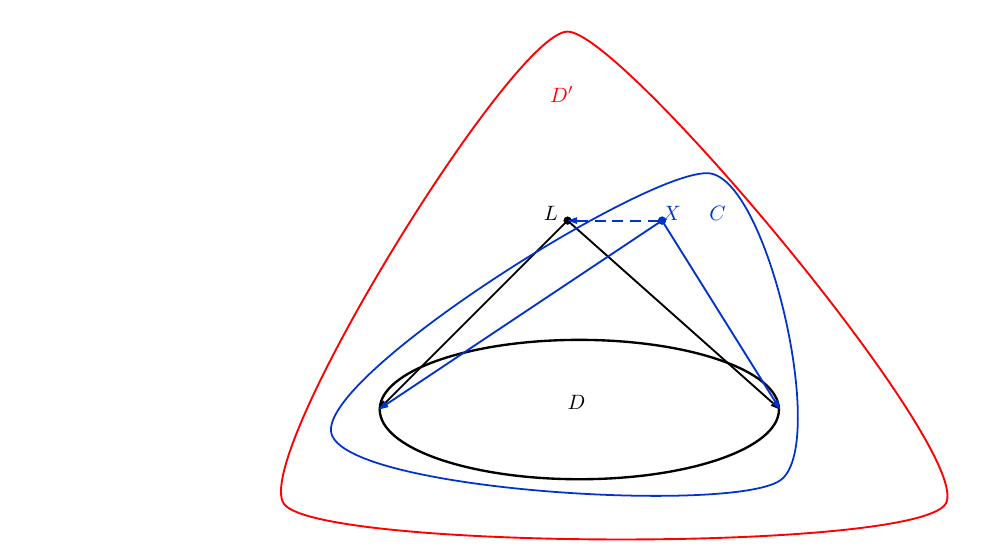}
\caption{One extra limit}
\label{fig:rem_compositionsanduniversal}
\end{figure}
\end{remark}

\subsection{Mixed limit-colimit computations} 
We now discuss computations where we can use limits and colimits together, we call these \emph{mixed computations}. 

\begin{definition}
  A \emph{mixed computation} is a finite sequence $(D_0,\dots,D_s)$ of diagrams with $D_{i}: I_{i} \to U(\cC)$, where $D_0$ consists only of morphisms in $\cA$, and for each $i=1,\dots,s$, $D_{i}$ is obtained from $D_{i-1}$ by adding either the limit of a subdiagram, with the corresponding cone morphisms or colimit of a subdiagram with the corresponding cocone morphisms.
\end{definition}

Note that there is no constructivity assumption for mixed computations. To include it would have been too restrictive and would have prevented natural applications like gluing geometric objects already constructed.  

\begin{example}[Monotone Boolean Circuits]\label{ex_monotoneboolean} Let $\cB_n$ be the lattice of subsets of $\left\{ 0,1 \right\}^{n}$, that is a category whose objects are the subsets of $\left\{ 0,1 \right\}^{n}$, and with $\Hom_{\cB_n}(A,B) = \left\{ \iota \right\}$ if $A\subset B$ where $\iota:A \to B$ is the inclusion, and $\Hom_{\cB_n}(A,B)=\emptyset$ otherwise. Let $Z_i = \left\{ (x_1,\dots,x_n)\in\left\{ 0,1 \right\}^{n} \,\,|\,\,x_i=1 \right\}$. Let $\cA_n$ be the set of basic morphisms $\{\id_{Z_i}\,\,|\,\,i=1,\dots,n\}$. Let $\cB = \coprod_{n=1}^{\infty}\cB_n$ be the disjoint union of these categories and $\cA = \coprod_{n=1}^{\infty}\cA_n$. 
  
  We show that there is a correspondence between multi-output monotone Boolean circuits with $n$ inputs and mixed computations in $\cB_n$. Given a monotone Boolean circuit, consider the corresponding straight line program of with Boolean operations. Start a mixed computation in $\cB_n$ with a copy each of the subsets $Z_{i}$. These correspond to the input variables $z_1,\dots,z_n$ of the straight line program. Subsequent entries $z_{n+1},z_{n+2}\dots$ will correspond to newly constructed objects in the mixed computation. For each operation in the straight line program of the form $z_i = z_j \wedge z_k$, take the limit of the objects corresponding to $z_j$ and $z_k$; similarly, take the colimit for $z_i = z_j \vee z_k$.
  To make a straight-line program from a mixed computation, start with the input variables $z_1,\dots,z_n$ and add $k-1$ new $\wedge$ operations for each limit of $k$ objects and a $k-1$ new $\vee$ operations for each colimit of $k$ objects (ignoring the arrows does not change the limit/colimit). 
 
  Thus, mixed computations in $\cB_{n}$ are in direct correspondence with monotone straight-line programs or, equivalently, monotone Boolean circuits. 
\end{example}

\begin{example}\label{ex_cubical}
  Consider the category $\opcat{Top}$ of topological spaces. Let $I=\left[ 0,1 \right]$ be the unit interval, and let the basic morphisms consist of, $I \xto{\id} I$,$I\to pt$, $pt\xto{0} I$, and $pt \xto{1} I$. We can build all cubes using limits, and then can glue these using colimits to 
  construct
  many topological spaces. 
\end{example}

We will reconsider mixed computations when we look at the complexity of projective schemes.  

\subsection{Cost and complexity}
Let $c_0: \cA \to \mathbb{Z}_{\geq 0}$ be any function, considered as the \emph{cost} of the basic morphisms.  

\begin{definition}
  The \emph{cost} of the computation $(D_0,\dots,D_s)$ is the the number of steps plus the cost of the initial diagram $D_0$ consisting of basic morphisms, that is
  $$ c(D_0,\dots,D_s) = s + \sum_{f\in \operatorname{edges}(I_0)} c_0(D_0(f)). $$

\end{definition}
If $c_0$ is not specified, then we consider it to be the constant function $1$, so every basic morphism will have unit cost. This will be the case in almost every example we consider.  

\begin{definition}
\label{def:categorical-complexity}
  The \emph{limit}  (respectively, \emph{colimit}, respectively, \emph{mixed}) {complexity},
 $c^{\op{lim}}_{\cC, \cA,c_0}(D)$ (respectively, $c^{\op{colim}}_{\cC, \cA,c_0}(D)$, respectively, $c^{\op{mixed}}_{\cC, \cA,c_0}(D)$), of a diagram $D$ in a category $\cC$ is the cost of the limit (respectively, colimit, respectively, mixed) computation using basic morphisms $\cA$, that has the smallest cost among all such computations that compute $D$. 
When, $c_0$ is the constant function $1$, we will omit it from the subscript and just write
$c^{\op{lim}}_{\cC, \cA}(D)$, $c^{\op{colim}}_{\cC, \cA}(D)$ or
$c^{\op{mixed}}_{\cC, \cA}(D)$.
 
 \hide{
  The \emph{ constructive limit}  (respectively, \emph{constructive colimit}) {complexity}, 
  $c^{\clim}_{\cC, \cA,c}(D)$ (respectively, $c^{\ccolim}_{\cC, \cA,c}(D)$), of a diagram $D$ in a category $\cC$ is the cost of the constructive limit (respectively, constructive colimit) computation using basic morphisms $\cA$, that has the smallest cost among all such computations that compute $D$. 
 }
 
For a morphism $f: X \rightarrow Y$ in $\cC$, the complexity $c(f)$ of $f$ is the complexity of the corresponding diagram mapping two objects and a single morphism $X \xto{f} Y$. 
For an object $X$ in $\cC$, the complexity $c(X)$ of $X$ is the complexity of the diagram with one object, $X$. 
\end{definition}

\begin{example}[Gluing Simplices]\label{ex_simplicialcomplexes}
  Let $\cC = \opcat{Top}$ be the category of topological spaces and let $\cA$ be the set of of all face embeddings 
 $ \Delta_{n}\hookrightarrow \Delta_{m}$
 corresponding to each strictly increasing maps $[n] \rightarrow [m]$, 
 where $\Delta_{n}$ is the 
 standard
 $n$-simplex. Colimit computations correspond to gluing operations between simplices. The colimit complexity then measures how many simplices are needed to construct a given topological space by gluing.  
\end{example}

We now go back to considering $\cC = \opcat{Set}$ with the basic morphisms $\cA$ consisting of a single morphism $\op{id}: \left\{ 1 \right\} \to \left\{ 1 \right\}$.

\begin{proposition}[Colimit complexity of sets]\label{prop:sets}
In the category 
$\opcat{Set}$, 
let 
\[
\cA = \left\{ \op{id}: \left\{ 1 \right\} \to \left\{ 1 \right\} \right\}, c_0(\id)=1.
\] 
Then,
for any set finite set $S$, 
\[
c^{\op{colim}}_{\opcat{Set},\cA}(S) = \card(S)+1.
\]
\end{proposition}

\begin{proof}
  Since finite sets of equal size are isomorphic, a computation will compute $S$ if and only if it computes any set of cardinality equal to 
  $\card(S)$. As in Example~\ref{ex_set}, starting with $\card(S)$ copies of $\left\{ 1 \right\}$ and taking their colimit, we get a set of cardinality
   $\card(S)$. So, the complexity is bounded from above by $\card(S)+1$. 
  To see that this is the most efficient way of producing a set with $\card(S)$ elements, we use Lemma~\ref{lem:constructivecomputation} below, which states that if we only care about building a single object, then a colimit computation can be replaced by a single colimit on $D_0$ consisting of basic morphisms.
Since the identity on $\left\{ 1 \right\}$ is the only basic morphism in this
case, taking the colimit of $\card(S)$ copies of $\left\{ 1 \right\}$ is the most
efficient way to obtain an object isomorphic to $S$. 
\end{proof}

\hide{
  \begin{example}[Non-constructive colimit complexity]\label{ex_nonconstructiveexample}
The following example shows the difference between colimit computations and non-constructive computations. Consider the computation
  \begin{tabularx}{0.5\textwidth}{ll}
    \code{1.} & \code{\_,}$\left\{ 1 \right\} \xtonormal{\id} \left\{1
    \right\}$\code{,1} \\
    \code{2.} & \code{\_,}$\left\{ 1 \right\} \xtonormal{\id} \left\{1
    \right\}$\code{,2} \\
  \code{3.} & \code{colim(1,2)} \\
  \code{4.} & \code{colim(1,2)} \\
  \dots & \\
  \code{a+2.} & \code{colim(1,2)} \\
  \code{a+3.} & \code{colim(3,4,\dots,a+2)} \\
\end{tabularx}

This computation produces a set of size 
$2a$.
Observe that only the step $\code{a+3}$ is not constructive. To make it constructive, one would need to include \code{1} and \code{2} in the colimit in step \code{a+3}, which would make this colimit be a set with just two elements. 
\end{example}
}

\begin{example}[Non-constructive colimit complexity]\label{ex_nonconstructiveexample}
The following example shows the difference between colimit computations and non-constructive computations. Consider the computation
  \begin{tabularx}{0.5\textwidth}{ll}
    \code{1.} & \code{\_,}$\left\{ 1 \right\} \xtonormal{\id} \left\{1
    \right\}$\code{,1} \\
    \code{2.} & \code{\_,}$\left\{ 1 \right\} \xtonormal{\id} \left\{1
    \right\}$\code{,2} \\
  \code{3.} & \code{colim(1,2)} \\
  \code{4.} & \code{colim(1,2)} \\
  \code{5.} & \code{colim(3,4)} \\
  \code{6.} & \code{colim(3,4)} \\
  \vdots & \vdots\\
  \code{2a-1.} & \code{colim(2a-3,2a-2)} \\
  \code{2a.} & \code{colim(2a-3,2a-2)} \\
  \code{2a+1.}& \code{colim(2a-1,2a)} \\
\end{tabularx}

This computation produces a set of size $2^{a}$. 
Observe that the steps $\code{5}$ to $\code{2a+1}$ are not constructive. To make them constructive, one would need to include
$\code{1}, \ldots \code{2i-2}$
 in the colimit in step $\code{2i-1}, 3 \leq i \leq a+1$,
and include
$\code{1}, \ldots \code{2i-2}$
in the colimit in step $\code{2i}, 3 \leq i \leq a$.
This would produce in  step $\code{2a+1}$ a set with just two elements. 
\end{example}

\begin{example}[Non-constructive mixed complexity]\label{ex_nonconstructive-mixed-example}
Using mixed computations, one can construct a set of doubly exponential complexity.
Consider for example the following mixed computation.

 \begin{tabularx}{0.5\textwidth}{ll}
    \code{1.} & \code{\_,}$\left\{ 1 \right\} \xtonormal{\id} \left\{1
    \right\}$\code{,1} \\
    \code{2.} & \code{\_,}$\left\{ 1 \right\} \xtonormal{\id} \left\{1
    \right\}$\code{,2} \\
  \code{3.} & \code{colim(1,2)} \\
  \code{4.} & \code{colim(1,2)} \\
  \code{5.} & \code{lim(3,4)} \\
  \code{6.} & \code{lim(3,4)} \\
  \vdots & \vdots\\
  \code{2a-1.} & \code{lim(2a-3,2a-2)} \\
  \code{2a.} & \code{lim(2a-3,2a-2)} \\
  \code{2a+1.}& \code{lim(2a-1,2a)} \\
\end{tabularx}
It is easy to check that it produces a set of cardinality $2^{2^a}$.
So the mixed complexity of a finite set $\card(S)$ is $\bigO(\log\log\card(S))$.
\end{example}

\hide{
\begin{remark}\label{rem_setmixed}
It is possible to give an example similar to the above for mixed computations by alternating limits and colimits to get double exponential set size starting with no basic morphisms, so the mixed complexity of a finite set $\card(S)$ is $\bigO(\log\log\card(S))$.
\end{remark}
}
One important feature of categorical complexity is that it allows one to define complexity of not just single objects or even morphisms of a given category (equipped with a set of basic morphisms), but one has a notion of complexity of arbitrary 
(finite) diagrams of the category as well. This last notion has no analog in classical theory of computational complexity.

We illustrate this feature in the following simple example.

\begin{example}
\label{eg:subspaces}
We will consider the colimit complexity  (in the category $\opcat{Vect}_k$ where $k$ is a field)
of two diagrams consisting of the inclusion morphisms of  $n^2$ different subspaces of dimension two  in a $2n$-dimensional $k$-vector space.

In the first diagram the subspaces are assumed to be generic  (Case \eqref{itemlabel:eg:subspaces:1} below), and we prove that the colimit complexity of the diagram is in $O(n^3)$ (which agrees with the intuition that to specify
$n^2$ inclusions of two dimensional subspaces in a $2n$-dimensional vector space we needs to specify a matrix in $k^{2n \times 2n^2}$).

In the second diagram  the subspaces are in a special position (Case \eqref{itemlabel:eg:subspaces:2} below), and  we prove that the colimit complexity of the diagram is in $O(n^2)$ (intuitively, only $2n$ distinct columns appear in the  $2n \times 2n^2$ corresponding to the inclusion if the basis vectors are chosen properly). Categorical complexity helps in quantifying the distinction in the complexity of the two diagrams having the same shape. There is no analog in classical computational complexity, which deals mainly with membership questions  (and thus from the point of view of category theory complexities of objects rather than general diagrams) of this kind of distinction.

\hide{   
We consider the category $\opcat{Vect}_k$ equipped with the same set of basic morphisms
as defined in Example \ref{ex_kvect-1}. 
}

We consider the category $\opcat{Vect}_k$ of vector spaces over a field $k$ and let $\cA$ consist of the scalar multiplication morphisms $k \xto{c} k$ for each $c\in k$, the addition morphism $k^2 \xto{+} k$, 
  the two projections $\pi_1,\pi_2:k^{2} \to k$, and morphisms $0 \to k$, $k\to 0$ (as in Example \ref{ex_kvect-1}).
 
\begin{enumerate}[(a)]
\item
\label{itemlabel:eg:subspaces:1}
Let  $n \geq 2$, $S_1,\ldots,S_{n^2} \subset V$ be subspaces of a finite dimensional $k$-vector space $V$, 
and suppose that
$\dim V = 2n$,  $\dim S_i = 2, 1 \leq i \leq n^2$, and
$\dim (S_i \cap S_j) = 0, 1 \leq i < j \leq n^2$.

Consider the diagram shown in Figure \ref{fig:subspace}, 
where the $\phi_i,1 \leq i \leq n^2$, are the inclusion homomorphisms.

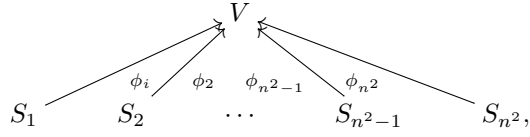
\begin{figure}
\[
\xymatrix{
&&V&& \\
S_1\ar[rru]_{\phi_i}& S_2\ar[ru]_{\phi_2} &\cdots & S_{n^2-1}\ar[lu]^{\phi_{n^2-1}}& S_{n^2} \ar[llu]^{\phi_{n^2}},
} 
\]
\caption{Subspace diagram}
\label{fig:subspace}
\end{figure}

We now show to produce the above diagram in Figure \ref{fig:subspace}
using a  colimit computation.

Let for each $i$, $1 \leq i \leq n^2$, 
$S_i',S_i''$ be one-dimensional subspaces of $S_i$ such that 
$S_i = S_i' \oplus S_i''$,
and let
$\psi_i',\psi_i''$ denote the inclusions of 
$S_i',S_i''$ 
into $V$ respectively.

Similarly, let 
$E_1,\ldots,E_{2n} \subset V$ be one-dimensional subspaces of $V$ 
such that $V = E_1 \oplus \cdots \oplus E_{2n}$.
Then, $V$ is canonically isomorphic to the direct product $E_1 \times \cdots \times E_{2n}$ and let $\pi_i:V \rightarrow E_i$
denote the canonical projections.

For each $1 \leq i \leq 2n, 1 \leq j \leq n^2 $ we first use colimit computation to obtain a diagram computing  the $2n^3$ linear morphisms,
$\phi_{i,j} = \pi_i \circ (\psi_j'\oplus \psi_j''):  S_j' \oplus S_j''  \rightarrow E_i$, to obtain the diagram shown in Figure \ref{fig:subspaces-new-1}.

The colimit computation to obtain the diagram shown in Figure \ref{fig:subspaces-new-1} is as follows. 
For each $j, 1\leq j \leq n$, first take the colimit of  $S_{j}'$ and $S_{j}''$ to obtain the diagram,
\[
\xymatrix{
S_j' \ar[rd]&& S_j''\ar[ld] \\
& S_j' \oplus S_j''& 
}
\]
and then successively for each $i,1 \leq i \leq n$,  take colimits of the diagrams
\[
\xymatrix{
&E_i& \\
S_j' \ar[rd]\ar[ru]^{\pi_i\circ\psi_j'}&& S_j''\ar[ld] \ar[ul]_{\pi_i\circ\psi_j''}\\
& S_j' \oplus S_j'' \ar@{.>}[uu]^{\phi_{i,j}}& 
}
\]
and observe (cf. Remark \ref{rem_compositionsanduniversal}) that we obtain the morphisms $\phi_{i,j}$.
Note that the morphisms $\pi_i\circ\psi_j', \pi_i\circ\psi_j''$, being morphisms between one dimensional $k$-vector spaces are
basic.
It is easy to check that the total cost of the colimit computation described above   is 
\[
2n^3 + n^2 + 2n^3 =  4n^3 + n^2.
\]

\hide{
Thus, while both arrangements of linear subspaces $L_1,\ldots, L_{n^2} \subset V$ contain the same number of subspaces, and would require the 
same number of polynomials to define (say using a first order formula in the language of the reals),   the second one is ``simpler''
than the first in terms of its categorical colimit complexity. 
}

\begin{figure}
\includegraphics[scale=0.7]{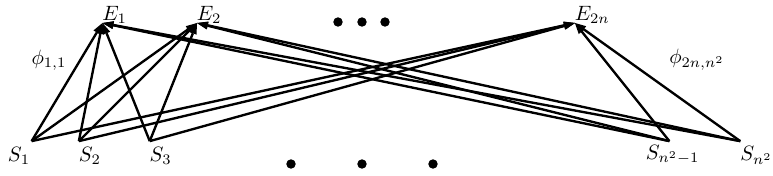}
\caption{Subspace diagram before  taking the last colimit}
\label{fig:subspaces-new-1}
\end{figure}

We then take the colimit of the whole diagram computed till now (as shown in Figure \ref{fig:subspaces-new-2}), and obtain the desired diagram as a subdiagram (whose morphisms are shown in blue in Figure \ref{fig:subspaces-new-2}).
\begin{figure}
\includegraphics[scale=0.7]{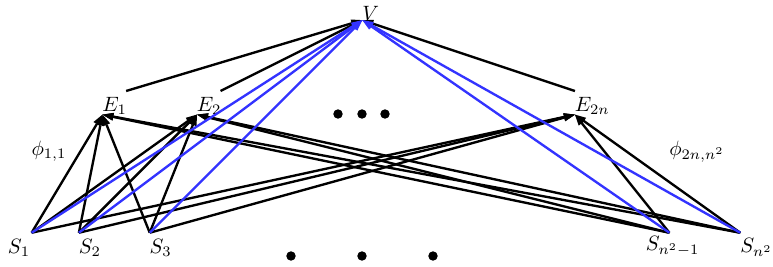}
\caption{Diagram after taking  the last colimit}
\label{fig:subspaces-new-2}
\end{figure}

The total cost is $4n^3+ n^2 + 1 = O(n^3)$.

\item
\label{itemlabel:eg:subspaces:2}
We now consider subspaces $S_1,\ldots,S_{n^2}$ subspaces in a special position, and prove that  
colimit  complexity of the 
corresponding diagram can be much smaller.

Let $V = V' \oplus V''$, where $\dim V'=  \dim V'' = n$.
Let $L_i' \subset V', L_i'' \subset V'', 1 \leq i \leq n$ be subspaces with $\dim L_i '= \dim L_i'' = 1, 1 \leq i \leq n$, and
suppose that  
$S_1,\ldots,S_{n^2}$ are the subspaces $L_i' \oplus L_j'' \subset V = V' \oplus V'', 1 \leq i,j\leq n$. 

Let $V'$ (respectively, $V''$) be the direct sum of one dimensional subspaces $E_1',\ldots,E_{n}'$
(respectively, $E_1'',\ldots,E_n''$). Then $V' = E_1' \oplus \cdots \oplus E_{n}', V'' = E_1''\oplus \cdots\oplus E_n''$, and
moreover, $V'$ (respectively, $V''$) is also canonically isomorphic to the direct product $E_1' \times \cdots \times E_{n}'$ 
(respectively, $E_1''\times \cdots \times E_n''$). Let for $1 \leq i \leq n$, 
$\pi_i':V' \rightarrow E_i'$ (respectively, $\pi_i'':V'' \rightarrow E_i''$)
denote the canonical projections. Also, for $1 \leq j \leq n$, let $\psi_j':L_j' \rightarrow V'$ (respectively, $\psi_j'':L_j'' \rightarrow V''$) denote the inclusion
morphisms, and denote $\phi_{i,j}' = \pi_i' \circ \psi_j'$ (respectively, $\phi_{i,j}'' = \pi_i'' \circ \psi_j''$).

In this case, we can construct the diagram in Figure \ref{fig:subspace}  in the following way.

\begin{figure}
\hspace{-1.5in}\includegraphics[scale=0.7]{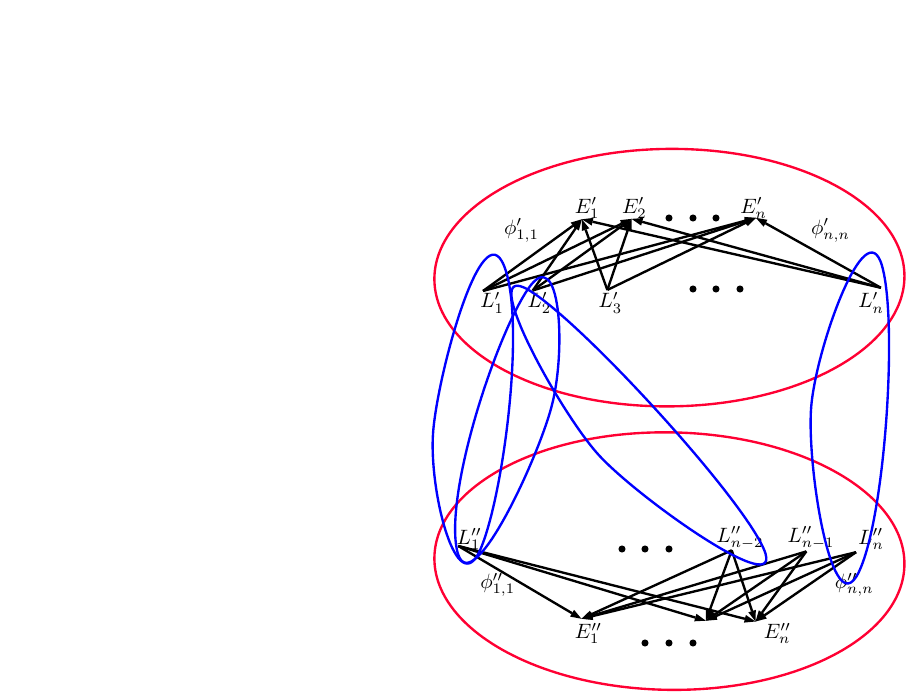}
\caption{Colimit computation of arrangement of subspaces in special position before taking any colimits.}
\label{fig:subspaces-new-3}
\end{figure}

\begin{figure}
\hspace{-1.5in}\includegraphics[scale=0.7]{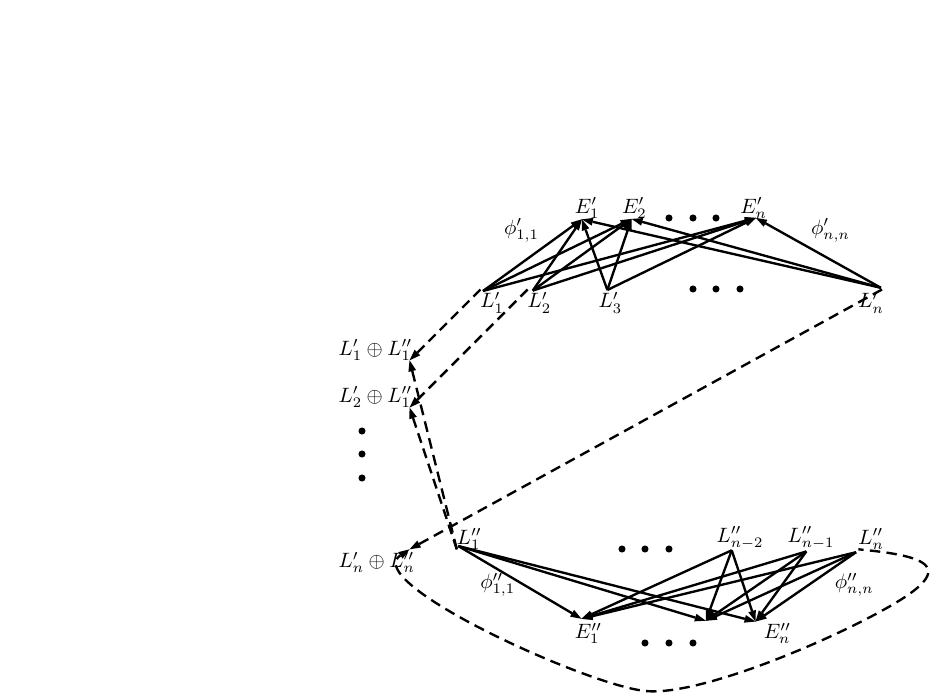}
\caption{Colimit computation of arrangement of subspaces in special position after taking $n^2$ colimits.}
\label{fig:subspaces-new-4}
\end{figure}

\begin{figure}
\hspace{-1.5in}\includegraphics[scale=0.7]{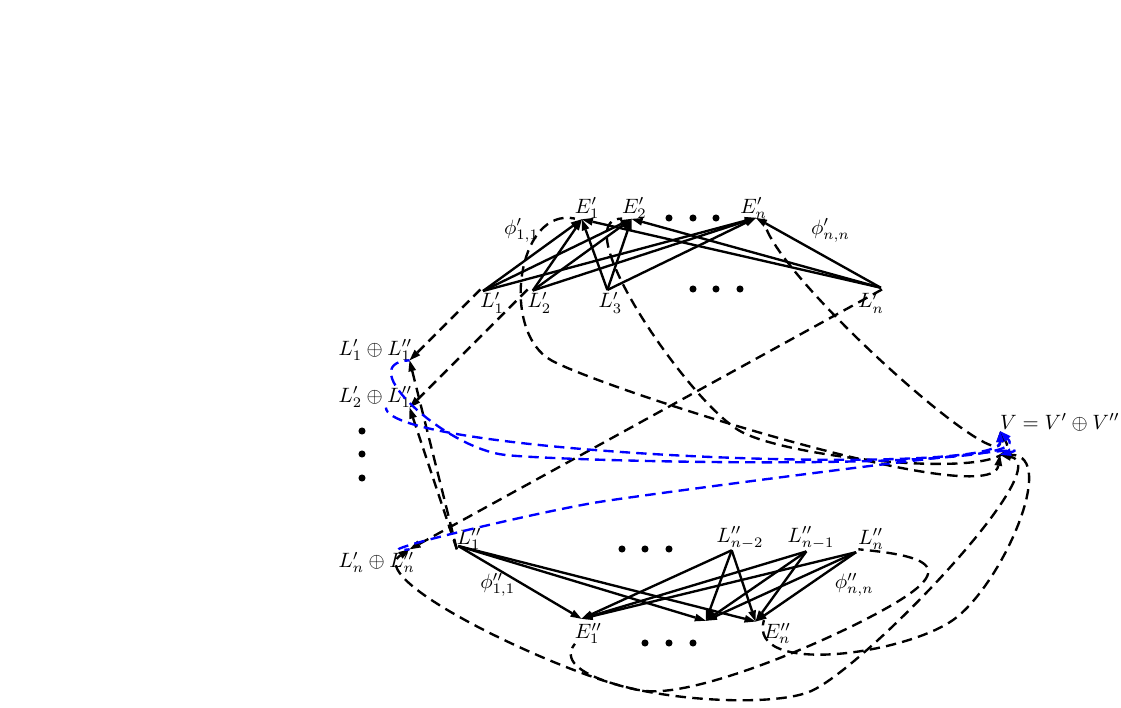}
\caption{Colimit computation of arrangement of subspaces in special position after taking the final colimit.}
\label{fig:subspaces-new-5}
\end{figure}

First construct 
the diagram containing the morphisms $\phi_{i,j}'$ and $\phi_{i,j}''$ as shown in Figure \ref{fig:subspaces-new-3}.
This costs $2n^2$.

Next take colimits of the $n^2$ subdiagrams circled in blue (corresponding to pairs of vertices $L_i,L_j', 1\leq i,j \leq n$) to obtain the diagram
shown in Figure \ref{fig:subspaces-new-4}. Finally, take a colimit of the entire diagram constructed so far to obtain the diagram 
shown in Figure \ref{fig:subspaces-new-5}. It is an easy exercise to check that the required diagram occurs as a subdiagram  (whose morphisms are shown using blue broken arrows in Figure \ref{fig:subspaces-new-5}) of the diagram so obtained.
The total cost is  $3n^2 +1 = O(n^2)$.
\end{enumerate}

Thus, while both arrangements of linear subspaces $S_1,\ldots, S_{n^2} \subset V$ 
in Parts \eqref{itemlabel:eg:subspaces:1} and \eqref{itemlabel:eg:subspaces:2}
contain the same number of subspaces, and would require the 
same number of polynomials to define (say using a first order formula in the language of the reals),   the second one is ``simpler''
than the first in terms of its categorical colimit complexity. 
\end{example}

\subsection{Useful facts about limit and colimit computations} We now collect a few facts that will be useful for proving statements about objects and morphisms computed by limit and colimit computations.  

The following lemma, which was already used in the proof of Proposition~\ref{prop:sets} above shows that, if the aim is to produce a specific object, intermediate steps in a limit or colimit computation are unnecessary. The key here is the constructivity assumption. 

\begin{lemma}
Assume $\cC$ has finite 
limits (respectively, colimits).
An object produced in a limit computation (respectively, 
colimit computation) is a limit (respectively, colimit) of a diagram consisting only of basic morphisms. 
  \label{lem:constructivecomputation}
\end{lemma}
  \begin{proof}
   Let $(D_0,\dots,D_s)$ be a limit computation and let $X$ be an
   object appearing in $D_s$. The point of the statement is that constructivity ensures that the information that would be added in intermediate limits is also included in the final limit that would produce $X$.  

  More precisely, let $L_i = \lim D_{i-1}|_{J_i}$ be the limit added to the diagram at the $i$th step. Let $J_{i}' = I_0 \cap J_{i}$. So we have that $D_{i-1}|_{J'_{i}}$ is the portion of the subdiagram of $D_{i-1}|_{J_i}$ which is also in $D_0$. We claim that $L_i \iso \lim D_{i-1}|_{J_i'}$. Indeed, the universal property of limits and constructivity imply that cones from any object $Z$ to $D_{i-1}|_{J_i'}$ can be uniquely extended to cones from $Z$ to $D_{i-1}|_{J_i}$, and therefore $\lim D_{i-1}|_{J_i'}$ satisfies the same universal property as $L_i$. 

The analogous proof holds for colimits.  
\end{proof}

The following remarks are very useful for working with objects and morphisms produced in limit computations. 

\begin{remark}
\label{rem:constructive-equalizer} If the category $\cC$ has finite products and equalizers, then we can write any limit $L$ as an equalizer, see e.g. \cite[5.4]{Awo}. 
More precisely, suppose that $J= (V,E,s,t)$ is a finite graph,
and
$D : J \to U(\cC)$ a diagram. 
Consider the following diagram:
\begin{equation}
\label{eqn:rem:contstructive-equalizer}
\xymatrix{
\prod_{v\in V} D(v)  \ar[rr]<-.8ex>_{\phi}  \ar[rr]<.8ex>^{\psi}&&\prod_{e\in E} D(t(e)) 
},
\end{equation}
where $\phi$ is the canonically defined morphism from $\prod_{v\in V} D(v)$ to the product
$\prod_{e\in E} D(t(e))$ induced from the morphisms 
\[
\left(\op{pr}_{t(e)}: \prod_{v\in V} D(v) \rightarrow t(e)\right)_{e \in E},
\]
and 
$\psi$ is the canonically defined morphism from $\prod_{v\in V} D(v)$ to the product
$\prod_{e\in E} D(t(e))$ induced from the morphisms 
\[
\left(D(e) \compose \op{pr}_{s(e)}: \prod_{v\in V} D(v) \rightarrow t(e)\right)_{e \in E},
\]
both maps  defined by the universal property of products.
\hide{
with $\phi_{D(t(e))} = \pi_{D(t(e))}$ 
and $\psi_{D(t(e))} = D(e)\compose
\pi_{D(s(e))}$, 
where for $v \in V$, \
\[
\pi_{D(v)}: \prod_{v\in V} D(v) \rightarrow D(v)
\] 
is the canonical projection.
}

Let $L \xtonormal{f} \prod_{v\in V} D(v)$ in the following diagram be the equalizer of the diagram \eqref{eqn:rem:contstructive-equalizer}:
 
 \[
\xymatrix{
L \ar[r]^{\hspace{-0.3in}f}& \prod_{v\in V} D(v)  \ar[rr]<-.8ex>_{\phi}  \ar[rr]<.8ex>^{\psi}&&\prod_{e\in E} D(t(e)) 
}.
\]

Then the object $L$ along with the morphisms
\[
\left(
\op{pr}_{D(v)}\circ f: L \rightarrow D(v)
\right)_{v \in V}
\] 
is isomorphic to the limit, $\lim D$, of $D$.
\end{remark}

\begin{remark}\label{rem:constructive-morphism-limits}
If $X\xto{f} Y$ is a morphism computed by a limit computation, and neither of $X$ and $Y$ is in $D_0$, then $X$ must have been computed as a limit of a diagram that contains $Y$. By constructivity, the diagram whose limit is $X$ must contain the diagram that produced $Y$ as a subdiagram. Therefore, if $X = \lim D$ where $D : (V,E)\to U(\cC)$ is a diagram of basic morphisms, then $Y = \lim D'$, where $D'$ is the subdiagram corresponding to a full subgraph $(V', E') \subset (V,E)$. We then have a commuting diagram
\[
\xymatrix{
X \ar[rr] \ar[d]^f && \prod_{v\in V} D(v)  \ar[rr]<-.8ex>_{\phi}  \ar[rr]<.8ex>^{\psi} \ar[d]_{\pi} &&\prod_{e\in E} D(t(e))\ar[d]_{\pi'}  \\
Y \ar[rr] && \prod_{v\in V'} D(v)  \ar[rr]<-.8ex>_{\phi'}  \ar[rr]<.8ex>^{\psi'}&&\prod_{e\in E'} D(t(e))  
}.
\]
In particular, $f$ is induced by the projection $\pi$. 
\end{remark}

We have corresponding facts for colimits as well. 
\begin{remark}
\label{rem:constructive-coequalizer} If the category $\cC$ has finite coproducts  and coequalizers, then we can write any colimit $M$ as an coequalizer. 
More precisely, suppose that $J= (V,E,s,t)$ is a finite graph,
and
$D : J \to U(\cC)$ a diagram. 
Consider the following diagram:
\begin{equation}
\label{eqn:rem:contstructive-coequalizer}
\xymatrix{
\coprod_{e\in E} D(s(e)) \ar[rr]<-.8ex>_{\phi}  \ar[rr]<.8ex>^{\psi} &&\coprod_{v\in V} D(v) 
},
\end{equation}
where $\phi$ is the canonically defined morphism from $\coprod_{e \in E} D(s(e))$ to 
$\coprod_{v\in V} D(v)$ induced from the morphisms 
\[
\left(\op{j}_{s(e)}: D(s(e)) \rightarrow  \coprod_{v\in V} D(v)\right)_{e \in E},
\]
and 
$\psi$ is the canonically defined morphism from $\coprod_{e \in E} D(s(e))$ to 
$\coprod_{v\in V} D(v)$ induced from the morphisms 
\[
\left(\op{j}_{t(e)} \compose D(e): D(s(e)) \rightarrow  \coprod_{v\in V} D(v)\right)_{e \in E},
\]
both maps  defined by the universal property of coproducts.

Let $\coprod_{v\in V} D(v) \xtonormal{f} M$ in the following diagram be the coequalizer of the diagram \eqref{eqn:rem:contstructive-coequalizer}:

\[
\xymatrix{
\coprod_{e\in E} D(s(e)) \ar[rr]<-.8ex>_{\phi}  \ar[rr]<.8ex>^{\psi} &&\coprod_{v\in V} D(v) \ar[rr]^{f} && M .
}
\]

Then the object $M$ with the morphisms
\[
(f \compose \op{j}_{D(v)}: D(v) \rightarrow M)_{v \in V}
\]
 is isomorphic to the to colimit, $\colim D$, of $D$.
\end{remark}

\begin{remark}\label{rem:constructive-morphism-colimits}
If $Y\xto{f} X$ is a morphism computed by a colimit computation, and neither of $X$ and $Y$ is in $D_0$, then $X$ must have been computed as a colimit of a diagram that contains $Y$. By constructivity, the diagram whose colimit is $X$ must contain the diagram that produced $Y$ as a subdiagram. Therefore, if $X = \colim D$ where $D : (V,E)\to U(\cC)$ is a diagram of basic morphisms, then $Y = \colim D'$, where $D'$ is the subdiagram corresponding to a full subgraph $(V', E') \subset (V,E)$. We then have a commuting diagram
\[
\xymatrix{
\coprod_{e\in E} D(s(e)) \ar[rr]<-.8ex>_{\phi}  \ar[rr]<.8ex>^{\psi} &&\coprod_{v\in V} D(v)  \ar[rr] && X  \\
\coprod_{e\in E'} D(s(e)) \ar[rr]<-.8ex>_{\phi}  \ar[rr]<.8ex>^{\psi}  \ar[u]^{\op{j}'}&&\coprod_{v\in V'} D(v) \ar[u]^{\op{j}} \ar[rr] 
 && Y\ar[u]^{f}
},
\]
where $\op{j},\op{j}'$ the morphisms induced by the universal property of coproducts.
In particular, $f$ is induced by the morphism $\op{j}$. 
\end{remark}

\section{Diagram computations, circuits and algebraic varieties}
\label{sec:circuits}
We now make a comparison between the arithmetic circuit complexity of polynomials and the limit complexity of the varieties which are their zero-sets. For simplicity, we will start with the category of affine algebraic varieties, but what we describe will make sense in other settings like affine schemes and algebras.

\subsection{Limit computations in affine varieties, schemes and $k$-algebras}
Let $\cC$ be the category $\AffVar_k$ of affine algebraic varieties over a field $k$. 
Let $\cA$ consist of the following basic morphisms. 
\begin{align}
  & \AA^{1} \xtonormal{c} \AA^{1}\t{, for each }c\in k ,\nonumber \\
  & \AA^{1}\times \AA^{1} \xtonormal{+} \AA^{1},\nonumber\\
  & \AA^{1} \times \AA^{1} \xtonormal{\times} \AA^{1},\label{eqn_affvar_basics}  \\
  & \AA^{1} \times \AA^{1} \xtonormal{\pi_1,\pi_2} \AA^{1},\nonumber\\
  & \AA^{1} \to \AA^{0}\t{, and }\AA^{0}\xtonormal{c} \AA^1\t{, for each } c \in k. \nonumber
\end{align}
Each of these morphisms is considered to have unit cost. 

\begin{example}  
As an example, let us make a limit computation of the morphism $\AA^{3} \xto{x^2+yz} \AA^{1}$ using these basic morphisms. 
\vspace{0.1in}

\vspace{0.1in}
\begin{minipage}[c]{0.5\textwidth}
  \hspace{0.3in}\includegraphics[width=2.0in]{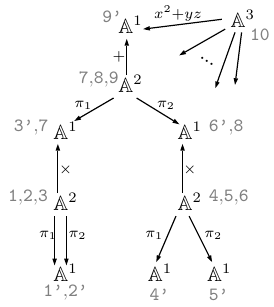}
\end{minipage}
\begin{minipage}[c]{0.5\textwidth}
\begin{tabularx}{0.5\textwidth}{ll}
  \code{1.} & \code{\_,}$\AA^{2} \xtonormal{\pi_{1}}\AA^{1}$\code{,1'} \\
  \code{2.} & \code{1,}$\AA^{2} \xtonormal{\pi_{2}}\AA^{1}$\code{,1'} \\
  \code{3.} & \code{1,}$\AA^{2} \xtonormal{\times}\AA^{1}$\code{,\_} \\
  \code{4.} & \code{\_,}$\AA^{2} \xtonormal{\pi_{1}}\AA^{1}$\code{,4'} \\
  \code{5.} & \code{4,}$\AA^{2} \xtonormal{\pi_{2}}\AA^{1}$\code{,\_} \\
  \code{6.} & \code{4,}$\AA^{2} \xtonormal{\times}\AA^{1}$\code{,\_} \\
  \code{7.} & \code{\_,}$\AA^{2} \xtonormal{\pi_{1}}\AA^{1}$\code{,3'} \\
  \code{8.} & \code{7,}$\AA^{2} \xtonormal{\pi_{2}}\AA^{1}$\code{,6'} \\
  \code{9.} & \code{7,}$\AA^{2} \xtonormal{+}\AA^{1}$\code{,\_} \\
  \code{10.} & \code{lim(1,2,3,4,5,6,7,8,9)}\\ 
\end{tabularx}
\end{minipage}
\vspace{0.1in}

The following shows that a similar computation can be done to compute any polynomial map. 

\end{example}

\begin{theorem}\label{thm_circuittodiagram}
  Let $f\in k\left[ x_1,\dots,x_n \right]$ be a polynomial of degree $d$. Assume that $f$ is computed by a straight line program (cf. \cite[Definition 2.1]{Burgisser-book2}) $\Gamma$
  of length $N$. Then, the limit-complexity (and therefore the categorical complexity) of the zero-set $X\subset \AA^{n}$ of $f$ is in $\bigO(N)$.
\end{theorem}

\begin{proof}
  Using the operations in the straight-line program $\Gamma$, we will construct a morphism $\AA^{n} \xto{f} \AA^{1}$. $X$ is then the limit $\AA^{n} \xto{f} \AA^{1} \xleftarrow{0} \AA^{0} $. 
  
  Associated to each straight-line program of length $s$, there is a morphism $(f_1,\dots,f_s) : \AA^{n} \to \AA^{s}$ where $f_1,\dots,f_s$ are the polynomials computed in each line of the diagram. Let $\Gamma = (\Gamma_1,\dots,\Gamma_N)$, where $\Gamma_i$ is the $i$th instruction in the straight line program $\Gamma$. For the first $n$ instructions of $\Gamma$, which introduce the variables $x_1,\dots,x_n$ as polynomials, we have the map $\AA^{n} \xto{\op{id}} \AA^{n}$. The space $\AA^{n}$ is constructed in a limit computation by taking the limit of $n$ disjoint copies of the diagram consisting of $\AA^{1}$ mapping to itself by $1$. To get the map $\AA^{n} \xto{\op{id}} \AA^{n}$, we take the limit of the whole diagram obtained so far. Costing $n+2$. Note that our diagram also contains all the projections from the source and target of $\op{id_{\AA^{n}}}$ to the $n$ components. This is the base case of the following inductive construction. 
  
  Assume that, for $n\leq k \leq N-1$, we have produced, in a limit computation, a morphism $(f_1,\dots,f_k) : \AA^{n} \to \AA^{k}$ corresponding to the polynomials computed by each line of $(\Gamma_1,\dots,\Gamma_k)$, together with the projections $p_1,\dots,p_k$ from $\AA^{k}$ to $k$ copies of $\AA^{1}$, all of which are in $D_0$. 
  
  Assume that $\Gamma_{k+1}$ is the multiplication of the $i$th and $j$th lines of the program, i.e. $f_{k+1} = f_i f_j$. We can then perform the following steps of the computation. For convenience, we have added, as if they were basic morphisms, the portion of the diagram used for the next steps as the first $k+1$ steps in this description: 

\begin{center}
  \includegraphics[width=3.3in]{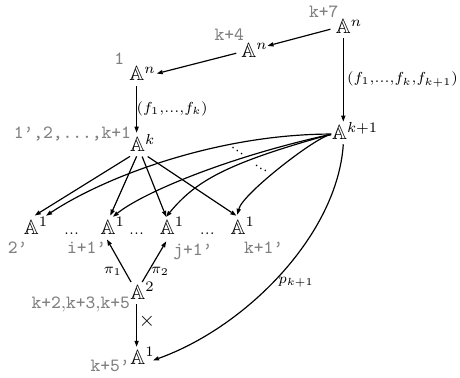}
\end{center}

\begin{tabularx}{0.5\textwidth}{ll}
  \code{1.} & \code{\_,}$\AA^{n} \xtonormal{(f_1,\dots,f_k)} \AA^{k}$ \code{,1'} \\
  \code{2.} & \code{1',}$\AA^{k}\xtonormal{p_1}\AA^{1}$\code{,2'} \\
  \code{3.} & \code{1',}$\AA^{k}\xtonormal{p_2}\AA^{1}$\code{,3'} \\
   & \dots \\
  \code{k+1.} & \code{1',}$\AA^{k}\xtonormal{p_k}\AA^{1}$\code{,(k+1)'} \\
  \code{k+2.} & \code{\_,}$\AA^{2}\xtonormal{\pi_1} \AA^{1}$\code{,(i+1)'} \\
  \code{k+3.} & \code{k+2,}$\AA^{2}\xtonormal{\pi_2} \AA^{1}$\code{,(j+1)'} \\
  \code{k+4.} & \code{lim((k+2),(i+1)',(j+1)', 1',1)} \\
  \code{k+5.} & \code{k+2,}$\AA^{2} \xtonormal{\times}\AA^{1}$\code{,\_} \\
  \code{k+6.} & \code{lim(2',3',\dots,(k+1)',(k+4)')}  \\
  \code{k+7.} & \code{lim((k+4), 1,1',2',3',\dots, (k+1)', k+2, (k+5)',k+6)}
\end{tabularx}

In the end, the map \code{(k+7)->(k+6)} is the map $\AA^{n} \xto{(f_1,\dots,f_{k+1})} \AA^{k+1}$, and the projection maps for the next step are the maps \code{(k+6)->2'},\code{(k+6)->3'},\dots,\code{(k+6)->(k+1)'}, \code{(k+6)->(k+5)'}. The process for addition steps is similar, with the step \code{k+5} modified. For scalar multiplication, it is similar with two steps less. For constants appearing in the computation, we add a new variable and take fiber product with $\AA^{0} \xto{c}\AA^{1}$ to fix the variable to the value $c$. Repeating this process until, k=N, we see that $(f_1,\dots,f_N)$ is produced in $\bigO(N)$ steps. We can then compose with the projection to the last coordinate, by taking a limit, to produce $\AA^{n}\xto{f} \AA^{1}$. To obtain the zero-set of $f$, we add $\AA^{0}\xto{0} \AA^{1}$ to this last $\AA^{1}$ and take the limit. 

It should also be noted that the way we made a limit computation for $\AA^{n}\xto{f}\AA^{1}$ was not the most efficient, but this way gives the cleanest inductive argument. For efficiency, the intermediate limit steps can be removed; cf. Lemma~\ref{lem:constructivecomputation}.
\end{proof}

\begin{remark}\label{rem_fcomp}
  The proof above shows that we can, by starting from a number of $\AA^{1}$ added as basic objects, construct any polynomial functions $f_1:\AA^{n}\to \AA^{1}$. We will use this fact later. 
\end{remark}

We now consider a converse for Theorem~\ref{thm_circuittodiagram} and show that, given a limit computation with low cost, an object $X$ computed by it is \emph{isomorphic to} the zero-set of a polynomial whose arithmetic circuit complexity is low. Since diagram computations produce objects up to isomorphism, categorical complexity does not reflect the complexity of every polynomial that might be used to cut out $X$ in a larger space. For example, $X$ could be the graph of a polynomial map $f : \AA^{n} \to \AA^{1}$ with very high arithmetic circuit complexity, but since $X$ would be isomorphic to $\AA^{n}$, its limit complexity would be very small, which does not say anything about the arithmetic circuit complexity of $f$. This is discussed in more detail in Section~\ref{sec:comparison_affvar}.

We first recall the definition of an arithmetic circuit over a field $k$, and that of the arithmetic circuit complexity of a polynomial map.
\begin{definition}[Arithmetic circuit complexity]
\label{def:arithmetic-circuit}
An arithmetic circuit $C$ over a field $k$ and variables $X_1,\ldots,X_n$ is a finite directed acyclic graph such that:
\begin{enumerate}[(i)]
\item
every vertex of the directed graph  with indegree $0$ is labeled either by a variable or an element of $k$;
\item
every other vertex   is labeled by either $+$ or $\times$.
\end{enumerate}
The \emph{size}  of an arithmetic circuit is the number of vertices  in the associated directed graph.
Each vertex  of $C$ is associated to a polynomial in $k[X_1,\ldots,X_n]$ computed at that  vertex (in the obvious way).
We say that \emph{a polynomial is computed by an arithmetic circuit $C$}, it if appears as the  polynomial associated to some vertex of the circuit.
For a tuple of polynomials $f=(f_1,\ldots,f_m) \in k[X_1,\ldots,X_n]$, we say that the \emph{arithmetic circuit  complexity}  of the induced polynomial map $f:k^n \rightarrow k^m$, is the size
of the arithmetic circuit of the smallest size that computes each of the polynomials $f_1,\ldots,f_m$. 
\end{definition}

\begin{theorem}\label{thm_diagramtocircuit}
  Let $(D_0,\dots,D_s)$ be a limit computation in in $\AffVar_k$, whose cost equals  $C=c(D_0,\dots,D_s)$. Then:
  \begin{enumerate}[(i)]
    \item
    \label{itemlabel:thm:diagramtocircuit:1} If $X$ is on object computed by $(D_0,\dots,D_s)$, then $X$ is isomorphic to the zero-set of a polynomial morphism $\AA^{m_1} \to \AA^{m_2}$ whose components are polynomials of degree at most $2$. The total arithmetic circuit complexity of the map is bounded above by $4C$. 
    \item
     \label{itemlabel:thm:diagramtocircuit:2}
     Every morphism $f:X\to Y$ in $D_s$ is a the restriction of a projection $\AA^{m_1} \to \AA^{m_1'}$ where $\AA^{m_1}$ and $\AA^{m_2}$ are the spaces from  
     Part  \eqref{itemlabel:thm:diagramtocircuit:1} where $X$ and $Y$ are embedded respectively. 
  \end{enumerate}
  
\end{theorem}

\begin{proof} 
If $s=0$, then the statement is true since the basic morphisms are  
  $\AA^{1} \xto{c} \AA^{1}$, 
  $\AA^{1}\times \AA^{1} \xto{+,\times} \AA^{1}$,
  $\AA^{1}\times \AA^{1} \xto{\pi_1, \pi_2} \AA^{1}$, 
  $\AA^{1} \to \AA^{0}$, and $\AA^{0}\xto{c} \AA^1$.  

  For the general case, constructivity implies that, as far as the isomorphism
  class of an object $X$ in $D_s$ is concerned, the intermediate limits in a
  limit computation can be removed (Lemma~\ref{lem:constructivecomputation}). Using the notation of the proof of Lemma~\ref{lem:constructivecomputation}, if $X = L_i$ is produced as a limit in the diagram computation, then $X \iso \lim D_{i-1}|_{J_i'}$ where $D_{i-1}|_{J'}$ is a subdiagram of $D_0$.  

  As discussed in Remark~\ref{rem:constructive-equalizer}, we can write $X = L_i$ as an equalizer. Let $J_i'= (V,E)$, $L_i \iso\lim D_{i-1}|_{J_i'}$; writing $D = D_{i-1}|_{J_i'}$, $L_i$ is isomorphic to the limit of 
\[
\xymatrix{
\prod_{v\in V} D(v)  \ar[rr]<-.8ex>_{\phi}  \ar[rr]<.8ex>^{\psi}&&\prod_{e\in E} D(t(e)).  
}
\]
where $\phi_{D(t(e))} = \pi_{D(t(e))}$ and $\psi_{D(t(e))} = D(e)\compose
\pi_{D(s(e))}$. 

 Since $J_i'\subset I_{0}$, we have that each $D(v)$ or $D(t(e))$ is either $\AA^{1}$, $\AA^{2}$, or a point; and $D(e)$ are addition, multiplication, constant, projection, or multiplication by a constant. So the above equalizer is of the form  
\[
\xymatrix{
\AA^{m_1} \ar[rr]<-.8ex>_{\phi}  \ar[rr]<.8ex>^{\psi}&&\AA^{m_2}
},
\]
and $L_i$ is isomorphic to the zero locus of $\AA^{m_1} \xto{\phi-\psi} \AA^{m_2} $. For the complexity of $\phi-\psi$, observe that $m_1,m_2\leq 2C$, and that each component of $\phi-\psi$ is a very simple polynomial which can be produced in two steps.

We now consider the second assertion. This follows directly from the discussion in Remark~\ref{rem:constructive-morphism-limits}. We observe that any morphism $f:X\to Y$ appearing in $D_s$ but not in $D_0$ must be a cone map from a limit $X = L_i = \operatorname{lim} D_{i-1}|_{J_i}$ to an object $Y$ appearing in $D_{i-1}|_{J_i}$. By the above construction, we know that $X$ and $Y$ are zero-loci of morphisms whose sources are products of objects in subdiagrams of $D_0$. Moreover constructivity implies that the subdiagram for $Y$ is contained in the subdiagram for $X$. The morphism $f$ is then the restriction of the projection from the subdiagram for $X$ to the subdiagram for $Y$. 
 
\end{proof}

The above arguments can also be considered for the category of affine schemes instead of varieties. Let $k$ be a field and let $\opcat{AffSch_{k}}$ be the category of affine schemes. Consider the same set $\cA$ of morphisms as in the case of affine varieties: 
$\AA^{1} \xto{c} \AA^{1}\t{, for each }c\in k$,
$\AA^{1}\times \AA^{1} \xto{+} \AA^{1}$, 
$\AA^{1} \times \AA^{1} \xto{\times} \AA^{1}$,
$\AA^{1} \times \AA^{1} \xto{\pi_1,\pi_2} \AA^{1}$,
$\AA^{1} \to \AA^{0}\t{, and }\AA^{0}\xto{c} \AA^1$.

One can also consider the category $\opcat{Alg}_{k}$ of $k$-algebras with basic morphisms
\begin{align*}
  & k\left[ x \right] \xtonormal{c} k\left[ x \right]\t{, for each }c\in k, \\
  & k\left[ x \right] \xtonormal{x\otimes 1 + 1 \otimes x} k\left[ x \right]\otimes k\left[ x \right], \\
  & k\left[ x \right] \xtonormal{x\otimes x} k\left[ x \right]\otimes k\left[ x \right], \\
  & k\left[ x \right] \xtonormal{id\otimes 1,\, 1 \otimes id} k\left[ x \right]\otimes k\left[ x \right],  \\
  & k\xtonormal{1} k\left[ x \right] \t{, and }\\
  & k\left[ x \right] \xto{x\mapsto c} k\t{, for each }c\in k. 
\end{align*}
These correspond to the basic morphisms considered for $\opcat{AffVar}_{k}$ under the adjoint equivalence
\[
\xymatrix{
\opcat{AffSch}_{k}   \ar@<1.0ex>[r]^{k\left[ \cdot \right]} & \opcat{Alg}_{k}^{\op{op}}  \ar@<1.0ex>[l]^{\,\,\,\,\op{Spec}}\\ 
}.
\]

Without any modification to the proofs, Theorems~\ref{thm_circuittodiagram} and~\ref{thm_diagramtocircuit} hold when $\opcat{AffVar}_{k}$ is replaced by $\opcat{AffSch}_{k}$ or $\opcat{Alg}_{k}$.

\subsection{Mixed computations in projective schemes} Let $\cC$ be the category $\opcat{Sch}_{k}$ of all schemes over a field $k$. Letting $\cA$ consist of the morphisms above as in the affine scheme case, we get a definition of complexity in the category of schemes.

The morphisms in $\cA$ are actually enough to produce all projective schemes using mixed computations. For example, in order to make $\PP^{1}$, we can produce the following diagram as a subdiagram of a computation and take its colimit:
\[
\xymatrix{
\AA^{1} && \AA^{1} \\
&Z(xy-1)\subset \AA^{2}\ar[lu]^{\pi_1} \ar[ru]_{\pi_2}&}
\] 
Note that $Z(xy-1)$ is isomorphic to $\AA^{1}-\left\{ 0 \right\}$. 

\begin{proposition}
  Let $X \subset \PP^{n}_{k}$ be the zero-scheme of homogeneous polynomials $f_1,\dots,f_m \in k\left[ x_0,\dots,x_n \right]$. Assume that, for each $i$, there is an arithmetic circuit of size $c_i$ computing $f_i$. Then, the categorical complexity of $X$ is in $\op{O}(n^2(c_1+\dots+c_s))$. 
\end{proposition}

\begin{proof}
  We first show how to construct $\PP^{n}$ by a mixed computation and then modify the construction to make the zero-scheme $X$ of $f_1,\dots,f_m$. 
 
  The last step in making $\PP^{n}$ will be to take the colimit of the diagram

\[\xymatrixcolsep{1pc}
\xymatrix{
A_0 = \AA^{n} && A_1 = \AA^{n} && A_2 = \AA^{n} && \dots && A_n = \AA^{n} \\
&Z_{0,1}\ar[lu] \ar[ru]  & Z_{0,2}\ar[llu] \ar[rru] & Z_{1,2}\ar[lu] \ar[ru]  &\dots & Z_{i,j} &\dots&Z_{n-1,n}\ar[ru]  &
}
\] 

There are $n+1$ copies of $\AA^{n}$ in the first row, which correspond to the standard covering  $U_0,\dots,U_n$ of $\PP^{n}$ by affine opens. In the second row, there is an object $Z_{i,j}$ for each pair $(i,j)$ with $i<j$; with each $Z_{i,j}$ corresponding to the intersections $U_i \cap U_j$ of the affine opens in the chart, each therefore isomorphic to the complement of a hyperplane in $\AA^{n}$. 

We will construct each $Z_{ij}$ as a subscheme of $\AA^{n}\times \AA^{n}=A_i \times A_j$, considered with coordinates $x_1,\dots,x_n,y_1,\dots,y_n$, defined by the equation:
\begin{equation}
  y_i x_j - 1 = 0,
  \label{eqn_zij1}
\end{equation}
and, for each $l\in \{1,\dots,$ ${i-1},{i+1}, \dots,n\}$, the equations 
\begin{equation}
  y_i x_l - y_l=0.
  \label{eqn_zij2}
\end{equation}
These equations describe the graph of the transition maps 
$$(x_1,\dots,x_n) \mapsto (
\frac{x_1}{x_j},
\frac{x_2}{x_j}, \dots, 
\frac{x_{i-1}}{x_j},
\frac{1}{x_j},
\frac{x_{i+1}}{x_j},\dots,
\frac{x_{j-1}}{x_j},
\frac{x_{j+1}}{x_j},\dots,
\frac{x_n}{x_j})$$ 
between the affine opens in the standard covering of $\PP^{n}$. Here, each affine open $U_i$, corresponding to points in homogeneous coordinates $[ \frac{a_0}{a_i}:\dots:\frac{a_{i-1}}{a_i}:1:\frac{a_{i+1}}{a_i}:\dots:\frac{a_{n}}{a_i} ]$ is parameterized by simply omitting the $i$th variable. The maps $Z_{i,j} \to A_i$ and $Z_{i,j} \to A_j$ are the restrictions of the projection maps from $A_i \times A_j$. 

To make the computation, start with $(n+1)$ sets of $n$ copies of $\AA^{1}$. Make $A_i$ as the limit of the $i$th set of $n$ $\AA^{1}$'s. Using the procedure described in the proof of Theorem~\ref{thm_circuittodiagram} (cf. Remark~\ref{rem_fcomp}), construct each $Z_{i,j}$ as the zero-set of the equations (\ref{eqn_zij1}) and (\ref{eqn_zij2}). Making the projections to the $A_i$ (cf. second part of Remark~\ref{rem_compositionsanduniversal}), we get the diagram above. At this point, the colimit of this diagram can be taken to produce $\PP^{n}$. 

To make the zero-scheme $X$ of $f_1,\dots,f_m$ in $\PP^{n}$, we continue in order to make the following diagram. 
\[\xymatrixcolsep{1pc}
\xymatrix{
X_0 && X_1  && X_2  && \dots && X_n   \\
&Z'_{0,1}\ar[lu] \ar[ru]  & Z'_{0,2}\ar[llu] \ar[rru] & Z'_{1,2}\ar[lu] \ar[ru]  &\dots & Z'_{i,j} &\dots&Z'_{n-1,n}\ar[ru]  &
}
\] 
where the $X_i$ are isomorphic to $X\cap U_{i}$ and the $Z'_{i,j}$ are isomorphic to $X \cap U_i\cap U_j$. To make the $X_i$, de-homogenize $f_1,\dots,f_m$ for each $A_{i}$ so that we get the equations for the subscheme of $A_i$ that corresponds to $X\cap U_i$. Using the procedure in Theorem~\ref{thm_circuittodiagram}, take the zero-scheme of these de-homogenized equations in each $A_{i}$. To make the $Z'_{i,j}$, make the map $X_i \times X_j \to A_i \times A_j$ and pull back $Z_{i,j}$. Finally, take the colimit of the diagram above to get a scheme isomorphic to $X$.

\end{proof}

\section{Categorical complexity of morphisms vs circuit complexity of polynomials}
\label{sec:comparisons}
While the motivation behind studying categorical complexity in various categories is a natural problem in its own right, it is still interesting to compare this new notion
with pre-existing notions of computational complexity in certain special categories. For pre-existing notion of complexity, the closest in spirit (in the categories we consider below) is that of arithmetic circuit complexity of polynomials. 
The goal of this section is to see how closely we can recover arithmetic circuit complexity (defined below in Definition \ref{def:arithmetic-circuit} ) (or equivalently straight-line program complexity).  We remark that arithmetic circuit complexity is
one of the
the most widely studied framework for non-uniform computational complexity theory (see for example \cite{Burgisser-book2}).

We discuss three categories where we can compare the arithmetic circuit complexity of a polynomial $f\in k\left[ x_1,\dots,x_n \right]$ with the categorical complexity of a morphism diagram. In the first one, we look at the morphism diagram $\op{D}_f = ({\AA^{n} \xto{f} \AA^{1}})$ in $\AffVar_k$ and see that its complexity can be very different than the arithmetic circuit complexity of $f$. The second category, the category of graded pairs of algebras, is an attempt at removing this discrepancy. The third category is the category of modules over $k\left[ x_1,\dots,x_n \right]$ where we prove concrete comparison results between the arithmetic circuit complexity of $f$ and the categorical complexity of the morphism diagram $k\left[ x_1,\dots,x_n \right] \xto{1\mapsto f} k\left[ x_1,\dots,x_n \right]$.

\subsection{Complexity of polynomial morphisms in $\AffVar_{k}$}\label{sec:comparison_affvar} We consider the category $\AffVar_{k}$ with the basic morphisms discussed above (\ref{eqn_affvar_basics}). Given a polynomial 
\[
f\in k\left[ x_1,\dots,x_n \right],
\] 
what is the categorical complexity of the diagram $\op{D}_f = ({\AA^{n} \xto{f} \AA^{1}})$ in $\AffVar_{k}$? 

First observe that since we have $\AA^0 \xtonormal{0}\AA^1$ as a basic morphism, we can use the pullback of the diagram
\[
\xymatrix{ \AA^n \ar[r]^f & \AA^1 & \AA^0 \ar[l]_{0} }
\]
to obtain the inclusion morphism $V \hookrightarrow \AA^n$ where $V = Z(f)$ is the affine subvariety of $\AA^n$  defined by $f$.
Thus, 
\[
c^{\lim}_{\AffVar_k,\mathcal{A}}(Z(f)) \leq c^{\lim}_{\AffVar_k,\mathcal{A}}(\op{D}_f) + 2,
\]
where $\mathcal{A}$ is the set of basic morphisms given in \eqref{eqn_affvar_basics}.

Categorical computations produce diagrams up to isomorphism, and categorical complexity is defined for isomorphism classes of diagrams. So, the complexity of $f$ is equal to the complexity of any other $\AA^{n} {\xto{g}} \AA^{1}$ where there is an automorphism $\phi: \AA^{n} \to \AA^{n}$ such that
\[
\xymatrix{
\AA^{n} \ar@{->}[d]_{\phi} \ar[r]^{f} & \AA^{1} \ar@{->}[d]^{\op{id}}\\
\AA^{n} \ar[r]^g & \AA^{1} 
}
\]
commutes. Therefore, the complexity of $\op{D}_f$ in $\AffVar_k$ is invariant under polynomial automorphisms of $\AA^{n}$
(and  the complexity of the variety $Z(f)$ is equal to the of $\phi(Z(f))$ for any polynomial automorphism $\phi$ of $\AA^n$).

For example, let $p\in k\left[ x_1,\dots,x_n \right]$ be any polynomial, and let $g\in k\left[ x_1,\dots,x_n,x_{n+1} \right]$, $g=x_{n+1} + p(x_{1},\dots,x_{n})$. Let $f=\pi_{n+1}$ sending 
$(x_1,\ldots,x_{n+1})$ to $x_{n+1}$, and $\phi(x_1,\dots,x_{n+1}) = (x_1,\dots,x_{n},x_{n+1} + p(x_1,\dots,x_n))$. Then, we have the following commuting diagram
\[
\xymatrix{
\AA^{n+1} \ar[r]^{\pi_{n+1}} & \AA^{1} & \,\,\,\,\,\,\,\,\,\,\,\,\,\,\,\,\,\,\,\, \,\,\,\,\,\,\,\,\,\,\\
\AA^{n+1} \ar@{->}[u]^{(x_1,\dots,x_n,x_{n+1}+p(x_1,\dots,x_n))} \ar[ur]_{x_{n+1}+p(x_1,\dots,x_n)} & \, & \,\,\,\,\,\,\,\,\,\,\,\,\,\,\,\,\,\,\,\,\,\,\,\,\,\,\,\,\,\,  
}
\]
So, the diagram $\op{D}_{x_{n+1}+p(x_1,\dots,x_n)}$ is isomorphic to $\AA^{n+1}\xto{\pi_{n+1}}\AA^{1}$. So, while the circuit complexity of $x_{n+1}+p(x_1,\dots,x_n)$ can be very high (for example, $p$ could be the permanent, or worse, a generic polynomial), the limit complexity of  $\op{D}_{x_{n+1}+p(x_1,\dots,x_n)}$ as well as that of $V(x_{n+1}+p(x_1,\dots,x_n))$ is bounded by $O(n)$.
This is because, geometrically, the zero-set of $x_{n+1} + p(x_1,\dots,x_n)$ is very simple, it is the graph of $-p$,  and is therefore isomorphic to $\AA^{n}$. 

It should also be noted that reductions in circuit complexity do not immediately lead to reductions in the complexity in $\AffVar_{k}$. 
Even though one can construct an arithmetic circuit for the polynomial $p(x_1,\dots,x_n)$ given one for 
$x_{n+1} + p(x_1,\dots,x_n)$ with a constant increase in size,  
this does not lead to an easy diagram computation of $\op{D}_{p(x_1,\dots,x_n)}$ from a computation of $\op{D}_{x_{n+1} + p(x_1,\dots,x_n)}$.
Still, one can ask whether polynomials which are believed to be hard to compute in 
the arithmetic circuit
model also have 
high categorical complexity. For example:
\begin{question}
\label{question:permanent}
  Is the limit/mixed complexity of $\AA^{n^2}\xto{\op{perm}_{n}}\AA^{1}$ polynomially bounded in $n$?
\end{question}

\subsection{Complexity of pairs of graded algebras}
It is possible that
the difference between complexity of $\op{D}_f$ in $\AffVar_k$ and the 
arithmetic circuit
complexity of $f$ is caused by the large number of automorphisms of $\AA^{n}$.
The goal of this section is to consider a category where arithmetic circuit  complexity is possibly close to categorical complexity,
and we can hope to bound the arithmetic circuit complexity of a polynomial $f$ from above by a polynomial function of the 
categorical complexity of the corresponding morphism (or that of the algebraic variety defined by $f$), and thus match the result
in Part \eqref{itemlabel:thm:diagramtocircuit:1} of Theorem \ref{thm_diagramtocircuit} which contains an inequality in the other direction. 

The automorphism group of $\PP_k^n$ is much smaller than that of $\AA^n_k$. Indeed, 
$\op{Aut}(\PP_k^n) \cong PGL(n+1,k)$, where $PGL(n+1,k)$ denotes the group of invertible $(n+1) \times (n+1)$ matrices with entries in $k$ modulo multiplication by $k^*$ (see for example \cite[Chapter 0, \S 5 b)]{Mumford}). Thus, the automorphisms of $\PP_k^n$ are all linear,
and hence
the complexity of a projective hypersurface $Z(f) \subset \PP^N_k$ should be closer to the  
 arithmetic circuit
 complexity of its defining polynomial $f$. 
However, since the basic objects are not projective, it is better to consider the corresponding setup in the category of graded algebras.

To this end, we could consider the category $\GA_k$ of graded algebras 
(given  a graded $k$-algebra $S$, the projective scheme $\op{Proj} S$ will be the corresponding geometric object).
The following set of morphisms will be taken to be the basic morphisms.

\begin{align}
\label{eqn:graded-basic}
\nonumber
  k[z^n] \xtonormal{+}  k[x^n,y^n]  &     \t{, }   z^n \mapsto x^n+y^n, \\ \nonumber
k[z^{2n}] \xtonormal{\cdot}  k[x^n,y^n] & \t{, }   z^{2n} \mapsto x^ny^n,\\ \nonumber
k[z^n]\xtonormal{i_1}   k[x^n,y^n]   &    \t{, }   z^n \mapsto x^n,\\  \nonumber
k[z^n] \xtonormal{i_2}   k[x^n,y^n]   &      \t{, }   z^n \mapsto y^n, \\ \nonumber
k[z^n] \xtonormal{c \times}   k[z^n]  &      \t{, }   z^n \mapsto c z^n, c \in k, \\ 
k[z^n] \xtonormal{}  k          &             \t{, } z^{n}\mapsto 0.
\end{align}

Note that $k[x^n,y^n] \cong k[x^n] \otimes_k k[y^n]$ is the coproduct of $k[x^n]$ and $k[y^n]$ in $\GA_k$, and hence using
colimits we can build graded polynomial rings in any number of variables. Also, the last morphism in the list \eqref{eqn:graded-basic}
allows us ``geometrically speaking''  to build the inclusion of a variety $V \hookrightarrow \PP^N_k$, defined by a set of homogeneous polynomials of the same degree. More precisely, in terms of graded algebras, if $V$ is defined by the ideal $\mathfrak{a}$ with generators
$g_1,\ldots,g_M \in k[z_0,\ldots,z_N]$ homogeneous of the same degree,   then the morphism
$k\left[ z_0,\dots,z_N\right] \rightarrow k\left[z_0,\dots,z_N\right]/\mathfrak{a}$ can be constructed by taking the  colimit of  the diagram
\begin{equation}
\label{eqn:colimit}
\xymatrix
{
k & k[ y_0,\ldots,y_M] \ar[l]_{0\hspace{.3in}}  \ar[r]^{y_i \mapsto g_i}&k[ z_0,\ldots z_N],
}
\end{equation}
noticing that the morphism $k[ y_0,\ldots,y_M] \xtonormal{0} k$ can be built from the basic morphism
$k[z] \xtonormal{}  k , z\mapsto 0$ (cf. Eqn. \eqref{eqn:graded-basic}) by taking colimits.

\begin{remark}
\label{rem:proj-is-not-a-functor}
We should  remind the reader here of one unpleasant aspect of the translation between algebra and geometry 
via taking $\op{Proj}$ of graded rings.
Unlike the functor 
\[
\op{Spec}: \opcat{Alg}_{k}^{\op{op}} \rightarrow \opcat{AffSch}_{k} 
\]
the map 
\[
\op{Proj}: \GA_{k}^{\op{op}} \rightarrow  \opcat{ProjSch}_k
\]
is not a functor, since to a morphism of graded $k$-algebras one can associate only a partial map between the 
corresponding projective schemes, and for graded $k$-algebras $R$ and $R'$, $\op{Proj} \; R$ and $\op{Proj}\; R'$ could be isomorphic, without $R$ being isomorphic to $R'$. Thus, isomorphisms of projective varieties considered in this section do not 
necessarily correspond exactly to isomorphisms of the corresponding graded algebras. It might be possible to consider as our objects
equivalence classes of graded algebras $S$ giving isomorphic $\op{Proj}\; S$, but we will not attempt to do this here. 
\end{remark}

Now consider for a tuple of homogeneous polynomial $f = (f_0,\ldots,f_N) \in k\left[ x_0,\dots,x_n \right]$, the diagram 
$$k[z_0,\ldots,z_N]\xtonormal{z_i\mapsto f_i}k\left[ x_0,\dots,x_n \right]$$ denoted also by $\op{D}_f$. 

However, the complexity of $\op{D}_f$ could still possibly be different from the 
arithmetic circuit 
complexity of $f$. The problem arises if morphism $\PP^n_k \rightarrow \PP^N_k$ induced by $f$ is an embedding of varieties.
More precisely, an embedding $ \nu:\PP_k^n = \op{Proj}k\left[ x_0,\dots,x_n\right]  \rightarrow \PP_k^N = \op{Proj}k\left[ z_0,\dots,z_N\right]$. 
In terms of graded algebras, this would correspond to a morphism $\op{D}_f: k\left[ z_0,\dots,z_N\right] \rightarrow k\left[ x_0,\dots,x_n\right]$
given by a tuple $f = (f_0,\ldots,f_N)$ of homogeneous polynomials of the same degrees, with $\op{D}_f$ being the map
$z_i \mapsto f_i$.
For example, $\nu$ could be the Veronese embedding of degree $d$, and in which case $N = \binom{n+d}{d}$, 
and $f$ would be the tuple all monomials
in $x_0,\ldots,x_n$ of degree $d$. But the inclusion $\nu(\PP_k^n) \hookrightarrow \PP_k^N$ also correspond to a 
morphism of graded algebras, namely 
$k\left[ z_0,\dots,z_N\right] \rightarrow k\left[z_0,\dots,z_N\right]/\mathfrak{a}$, where $\mathfrak{a}$ is the homogeneous ideal
of $\nu(\PP^n_k)$. Suppose that $\mathfrak{a}$ is generated by the homogeneous polynomials $g_0,\ldots,g_M$. In the case
when $\nu$ is a Veronese embedding, the polynomials $g_i$ can be taken to be a set of quadratic binomials. Now the morphism
$k\left[ z_0,\dots,z_N\right] \rightarrow k\left[z_0,\dots,z_N\right]/\mathfrak{a}$ can be constructed by taking the  colimit 
(see \eqref{eqn:colimit} above).

If the $g_i$'s have small colimit complexity compared to the  $f_i$'s then the second morphism will have smaller colimit
complexity. And thus the categorical complexity of the embedding $\nu:\PP^n_k \rightarrow \PP^N_k$ could be determined by 
a categorical computation involving only the polynomials $g_0,\ldots,g_M$, and in principle could be much smaller than the arithmetic circuit complexity of $f$.  

To prevent the phenomenon described above, which might cause the categorical complexity of a morphism 
$\op{D}_f$ to diverge from the arithmetic circuit complexity of $f$, we consider 
the category whose objects are pairs $(X,\PP^n_k)$, with $X$ a subvariety of $\PP^n_k$. Or rather, we consider the corresponding morphisms of the coordinate algebras.  
We denote by $\GAP_k$ the category whose objects are 
surjective morphisms $(A \xto{f} B)$, where $A$ is isomorphic to some 
polynomial ring $k[x_0,\ldots,x_n]$ graded by degree. For example, suppose that 
$A = k[x]$ and $B = k[x]/(x^2)$ both graded by degree, and $A \xto{f} B$ the canonical surjection.

In order to define categorical complexity in $\GAP_k$ we define the basic morphisms as follows. For each $n \geq 1$, we include in the set of basic morphisms the following set
of morphisms (all polynomial rings appearing below are graded by degrees and the degrees of the indeterminates $x,y,z,\ldots$ appearing in the morphisms are all equal to $1$):
\begin{align}
\label{eqn:graded-pairs-basic}
\nonumber
(k[z^n] \xtonormal{\sim} k[z^n]) \xtonormal{+} (k[x^n,y^n] \xtonormal{\sim} k[x^n,y^n])& \t{, }z^n \mapsto x^n+y^n, \\
\nonumber
(k[z^{2n}] \xtonormal{\sim} k[z^{2n}]) \xtonormal{\cdot} (k[x^n,y^n] \xtonormal{\sim} k[x^n,y^n])& \t{, }z^{2n} \mapsto x^ny^n,\\
\nonumber
(k[z^n] \xtonormal{\sim} k[z^n]) \xtonormal{i_1} (k[x^n,y^n] \xtonormal{\sim} k[x^n,y^n])& \t{, }z^n \mapsto x^n, \\
\nonumber
(k[z^n] \xtonormal{\sim} k[z^n]) \xtonormal{i_2} (k[x^n,y^n] \xtonormal{\sim} k[x^n,y^n])& \t{, }z^n \mapsto y^n, \\
\nonumber
(k[z^n] \xtonormal{\sim} k[z^n]) \xtonormal{c \times } (k[z^n] \xtonormal{\sim} k[z^n]) & \t{, }z^n \mapsto c z^n, c \in k, \\
(k[z^n] \xtonormal{\sim} k[z^n]) \xtonormal{}  (k[x^n] \xtonormal{(x\mapsto 0)} k)& \t{, }(z^n\mapsto x^n, z^n\mapsto 0).
\end{align}

\begin{remark}
\label{rem:gapk}
Notice that in all but the the last morphism in the above list, the pairs occurring as source and target of the morphism consist of
isomorphic objects, and the morphism between them is diagonal (i.e. of the form $(\phi,\phi)$ for some morphism $\phi$ in the
category $\GA_k$). Only in the last morphism this is not true, and the morphism between the pair in this case is $(\op{id}, \phi)$,
where $\phi$ is precisely the morphism used to build inclusion of varieties $\GA_k$ (cf. \eqref{eqn:colimit}). This remark will be used 
in what follows.   
\end{remark}

Notice that we can obtain the morphism
$$ (k[z^n] \xtonormal{(z^n\mapsto 0)} k) \xtonormal{(z^n\mapsto 0)} (k[x^n] \xtonormal{\sim} k[x^n]) $$ as the colimit diagram:

\[\xymatrixcolsep{3.8pc}
\xymatrix{
{(k[t^n] \xtonormal{\sim} k[t^n])} \ar[d]^{t^n \mapsto z^n, t^n \mapsto 0}\ar[r]^{t^n\mapsto 0,t^n\mapsto 0} & {(k[y^n] \xtonormal{\sim} k[y^n])} \ar[d]\\
{(k[z^n] \xtonormal{(z^n\mapsto 0)} k)} \ar[r]^{(z^n \mapsto 0)} & {(k[x^n] \xtonormal{\sim} k[x^n])}  
}
\]

Given a tuple of homogeneous polynomials $f=(f_1,\ldots,f_N)  \in k[x_0,\ldots,x_n]^{N+1}$, we consider the categorical complexity of the diagram $\op{M}_f$ defined as, 
\[
(k[z_0,\ldots,z_N] \xtonormal{\sim} k[z_0,\ldots,z_N]) \xtonormal{\op{M}_f} (k[z_0,\ldots,z_n] \xtonormal{\sim} k[x_0,\ldots,x_n]),
z_i \mapsto f_i.
\]
Following the same proof as Theorem~\ref{thm_circuittodiagram}, we have:
\begin{proposition}
\label{prop:GAP}
  Given a tuple of homogeneous polynomials 
  \[
  f = (f_0,\ldots,f_N) \in k\left[ x_0,\dots,x_n \right]^{N+1}
  \] having  arithmetic  circuit complexity $M$, the colimit complexity of the morphism diagram $\op{M}_f$ is in $\bigO(M)$.
\end{proposition}

We claim also that the phenomenon described previously in the non-embedded setting does not occur in this 
embedded pair setting.
Consider for example, 
a embedding $\nu:\PP_k^n \rightarrow \PP_k^N$ as before, but now consider 
the embedding of the pair $(\PP_k^n,\PP_k^n)$ into $\PP_k^N,\PP_k^N)$ by the diagonal morphism
$(\nu,\nu)$.
In terms of pairs of graded algebras, this would correspond to a morphism 
\[\op{M}_f =(\op{D}_f,\op{D}_f): (k\left[ z_0,\dots,z_N\right] \xtonormal{\sim}k\left[ z_0,\dots,z_N\right])\rightarrow (k\left[ x_0,\dots,x_n\right] \xtonormal{\sim} k\left[ x_0,\dots,x_n\right] ),
\]
given by a tuple $f = (f_0,\ldots,f_N)$ of homogeneous polynomials of the same degrees, with $\op{D}_f$ being the map
$z_i \mapsto f_i$.
Suppose also as before that  the inclusion $\nu(\PP_k^n) \hookrightarrow \PP_k^N$ also correspond to a 
morphism of graded algebras, namely 
$k\left[ z_0,\dots,z_N\right] \rightarrow k\left[z_0,\dots,z_N\right]/\mathfrak{a}$, where $\mathfrak{a}$ is the homogeneous ideal
of $\nu(\PP^n_k)$. Suppose that $\mathfrak{a}$ is generated by the homogeneous polynomials $g_0,\ldots,g_M$.  
However, unlike in the non-embedded case we cannot construct the morphism,
\[
(k\left[ z_0,\dots,z_N\right] \xtonormal{\sim} k\left[ z_0,\dots,z_N\right]) \rightarrow  (k\left[z_0,\dots,z_N\right]/\mathfrak{a} \xtonormal{\sim} k\left[z_0,\dots,z_N\right]/\mathfrak{a})
\]
using the polynomials $g_i$ and taking a colimit, since the morphism $k[z] \xtonormal{0} k$ used to construct the morphism
$k[ y_0,\ldots,y_M] \xtonormal{0} k$ in the 
colimit in \eqref{eqn:colimit}, is available only in the second slot of the pairs and not in the first 
(see Remark \ref{rem:gapk}).

We can now ask whether the converse of Proposition \ref{prop:GAP}  is true.

\begin{question}
\label{question:gap}
For a tuple of homogeneous polynomials $f = (f_0,\ldots,f_N) \in  k\left[ x_0,\dots,x_n \right]^{N+1}$ for which the colimit complexity of the morphism diagram $\op{M}_f$ is in $\GAP_k$ is $M$, is the arithmetic circuit complexity of $f$ 
polynomially bounded in $M$ ? 
\end{question}

\subsection{Arithmetic  circuit and categorical complexity of modules over polynomial rings}\label{sec:modules} 
We now consider the relationship between arithmetic circuit complexity and categorical complexity in categories of modules. 
Unlike in the categories of affine varieties and graded algebras (namely, $\Aff_k$, $\GA_k$, $\GAP_k$) considered in the last two sections,
we are able to relate polynomially the arithmetic circuit complexity and categorical complexity of polynomials (and their
induced morphisms in the module category) by proving inequalities in both directions. Thus, we are able to prove stronger 
relation between circuit complexity and categorical complexity -- though the category is  perhaps less interesting from the geometric point of view than the ones considered previously. 

Let $R = k\left[ x_1,\dots,x_n \right]$. We consider the category $R\t{-}\opcat{Mod}$ of $R$-modules. 
We consider colimit computations in $R\Mod$ with the following set of basic morphisms.

\begin{eqnarray}
\label{eqn:basic-RMod}
\nonumber
 &&  R \xtonormal{x_i} R, \t{ for each }i=1,\dots,n, \\
 \nonumber
 &&  R \xtonormal{c} R,\t{for each }c\in k, \\
 &&  R \xtonormal{i_1,i_2} R \oplus R, \\ 
 \nonumber
 &&  R \xtonormal{\Delta} R \oplus R, \\ 
 \nonumber
 &&  R \oplus R \xtonormal{+} R, \\
 \nonumber
 &&  R \to \left\{ 0 \right\}.
\end{eqnarray}

For a polynomial $f\in k\left[ x_1,\dots,x_n \right]$, we consider the corresponding morphism $R\xto{f}R$ that sends $1$ to $f$. Recall that a formula is an arithmetic circuit or straight line program where past intermediate computations cannot be re-used (see \cite[\S 2.2.1]{Burgisser-book2}).

\begin{proposition}
\label{prop:formula-vs-RMod}
  If a polynomial $f\in k\left[ x_1,\dots,x_n \right]$ is computed by a formula of size $s$, then the diagram $R\xto{f} R$ is computed by a colimit computation in $R\Mod$ with cost bounded by $\bigO(s)$. 
\end{proposition}
\begin{proof}
  Without loss of generality, assume that all sum and product gates have fan-in at most two. 
  We will build, for each formula $C$, a diagram $D_C$ whose colimit will contain $R \xto{p_c} R$ where $p_C$ is the output polynomial of $C$. This will be done inductively on the size of $C$.
  
 Each $D_C$ will be a diagram of the form   
\[\xymatrixcolsep{1pc}
\xymatrix{
R  \ar[r] &  \fbox{\parbox{1.3in}{$\,\,\,\,\,\,\,\,$\\\,\,\,\,\\\,\,\,\,}}\ar[r] & R  
}.
\]
whose colimit is $R$ with the morphism from the $R$ on the right to the colimit being $\id_{R}$ and the morphism from the $R$ on the left to the colimit being defined by $1 \mapsto p_C$. 

If the output $p_C$ of $C$ is one of the variables $x_i$ let $D_C$ be the diagram $R\xto{x_i} R$. If it just a constant, then $D_C$ is $R\xto{c}R$.

If the top gate of $C$ is a product gate with $C'$ and $C''$ as the left and right sub-circuits, then we set $D_C$ by chaining together $D_{C'}$ and $D_{C''}$: 
\[\xymatrixcolsep{1pc}
\xymatrix{
R  \ar[r] &  \fbox{\parbox{1.3in}{$\,\,\,\,\,\,\,\,$\\\,\,\,\,\\\,\,\,\,}}\ar[r] & R \ar[r] & \fbox{\parbox{1.3in}{$\,\,\,\,\,\,\,\,$\\\,\,\,\,\\\,\,\,\,}}\ar[r] & R  
}.
\]
The map from the left-most $R$ to the colimit is the composition $R \xto{p_{C'}} R \xto{p_{C''}} R$, which is $R \xto{p_{C'}p_{C''}} R$.  

If the top gate of $C$ is a sum gate with $C'$ and $C''$ as the left and right sub-circuits, then we define $D_C$ as
\[
\xymatrixrowsep{1pc}
\xymatrix{
\, & \, & R\ar[r]\ar[dl]_{i_1} &  \fbox{\parbox{1.3in}{\center $\,$\\$\,$}}\ar[r] & R\ar[dr]^{i_1} & &  \\  
R\ar[r]^{\Delta} & R\oplus R & & & & R\oplus R\ar[r]^{+} & R \\
\, & \, & R\ar[r]\ar[ul]^{i_2} &  \fbox{\parbox{1.3in}{\center $\,$\\$\,$}}\ar[r] & R\ar[ur]^{i_2} & &  \\  
}.
\]
where the top and bottom rows are $D_{C'}$ and $D_{C''}$. The colimit of this diagram is again $R$ with the map from the left-most $R$ to the colimit being $p_{C} + p_{C'}$.
\end{proof}

What about a converse? What does the existence of a colimit computation in $R\Mod$ that produces $R\xto{1\mapsto f} R$ say about the complexity of $f$?

\begin{theorem}\label{thm_rmoddiagramtocircuit}
  Let $R = k\left[ x_1,\dots,x_n \right]$. If $R\xto{f} R$ is computed in a colimit computation with cost  $c$ in $R\Mod$, then there is an arithmetic circuit of size $\op{poly}(c)$ with inputs $x_1,\dots,x_n$, that computes $f$.
\end{theorem}
\begin{proof}
First of all note that it is a standard fact that the ring $R$ is a unique factorization domain, and  hence
given any finite tuple $(f_1,\ldots,f_q)$ of polynomials in $R$, there exists a unique (up to multiplication
by units) of a greatest common divisor (GCD) of $f_1,\ldots,f_q$. We will need an algorithmic fact
related to computation of GCD in the proof below (cf. \eqref{itemlabel:GCD}).
  
Consider a diagram $D: I \to R\Mod$ consisting only of the basic morphisms described above 
in \eqref{eqn:basic-RMod}.  Assume that we have $\op{colim}{D}=R$. 
  
  For each vertex $v\in I$, we have that $D(v)$ is $R$, $R\oplus R$ or $\left\{ 0 \right\}$. For each $v$ such that $D(v)=R$, let $f_v$ be the image of $1$ under the morphism $R \xto{1\mapsto f_v} R$ from $D(v)$ to the colimit $R$. If $D(v)= \left\{ 0 \right\}$, then we set $f_v = 0$. If $D(v)= R\oplus R$, then we set two polynomials $f_v$ and $f_{v'}$ so that the map $R\oplus R \to R$ to the colimit is given by $(1,0) \mapsto f_{v}$ and $(0,1) \mapsto f_{v'}$. We will prove that each $f_v$ is computed by a polynomially sized circuit.  
  
  We are considering the $f_v$'s as unknowns in a system of equations.  For each arrow in $D$, we consider one or two $R$-linear equations. For an arrow $D(v_1) \to D(v_2)$ of the form given in the left column, we add the equations in the right column:  
\begin{eqnarray*}
  &  R \xtonormal{x_i} R                 & f_{v_1} - x_i f_{v_2}                          \\
  &  R \xtonormal{c} R                 & f_{v_1} - c f_{v_2}                            \\
  &  R \xtonormal{i_1,i_2} R \oplus R    & f_{v_1} - f_{v_2}\t{ or } f_{v_1} - f_{v_2'}   \\ 
  &  R \xtonormal{\Delta} R \oplus R     & f_{v_1} - f_{v_2} - f_{v_2'}                   \\ 
  &  R \oplus R \xtonormal{+} R          & f_{v_1} - f_{v_2} \t{ and } f_{v_1'}-f_{v_2}   \\
  &  R \to \left\{ 0 \right\}      & f_{v_1} = 0 \t{ and } f_{v_2} = 0              .
\end{eqnarray*}
In this way, we obtain a homogeneous system $k\left[ x_{1},\dots,x_{n} \right]$-linear equations; $A\vec{f}=0$, $A\in \op{Mat}_{n_2\times s}(R)$. Tuples $\vec{f} = (f_{v_1},f_{v_2}, f_{v_3},\dots, f_{v_s}) \in k\left[ x_1,\dots,x_n \right]^{s}$ that satisfy this system of equations correspond to a cocones of the diagram $D$ with target $R$. 

Since the colimit of $D$ is $R$, for any such cocone corresponding to $(f_{v_1},\dots,f_{v_s})$, there will be a map $(\colim D = R)\,\, \xto{1 \mapsto g}\,\,\, R$ making the diagram containing the new cocone, the colimit cocone and the map $R \xto{1 \mapsto g} R$ commute. This implies that $g$ divides each $f_{v_{j}}$. Since the colimit is the initial cocone, we can find the tuple of polynomials corresponding to the colimit cocone by taking $(\frac{f_{v_{1}}}{h},\frac{f_{v_2}}{h},\dots,\frac{f_{v_s}}{h})$ where $h = \op{gcd}(f_{v_1},\dots,f_{v_s})$. Thus, assuming the colimit is $R$, to compute the map from every $D(v)$ to the colimit, it suffices to: (\emph{i}) find a solution to the above system of equations for $D$, and (\emph{ii}) divide by $\op{gcd}(f_{v_1},\dots,f_{v_s})$. 

To solve the system $A\vec{f} = 0$, we use Gaussian elimination over the field $k(x_1,\dots,x_n)$. Let $R$ and $C$ be square matrices with entries in $k(x_1,\dots,x_n)$ such that $RAC$ is a diagonal matrix, in the sense that it contains an $r\times r$ minor which is $I_r$, with $r<s$, and all other entries are $0$. Following the steps of Gaussian elimination, there exist circuits (or straight-line programs) with division that produce each entry of $R$ and $C$ in time $\op{poly}(s)$. 
Also, observe that the entries of $R$ and $C$ have degree at most $s$.
Without loss of generality, assume $RACe_1 = 0$. Then 
$$Ce_1 = \left( \frac{p_1}{q_1},\dots, \frac{p_s}{q_s} \right)^{T} \in k(x_1,\dots,x_n)^{s}$$
is a solution to $A\vec{f} = 0$.  

Before we proceed further and get this solution into $k\left[ x_1,\dots,x_n \right]^{s}$, we make a small digression into algorithms about straight-line programs. We consider the following, which are both proven by the algorithms in Kaltofen's \cite{kaltofengcd}: 
\begin{itemize}
  \item[(GCD)] 
  \label{itemlabel:GCD}
  Greatest common divisor: Given polynomials $f_1,\dots,f_q \in k\left[ x_1,\dots,x_n \right]$ which are the output of a circuit of size $s'$, there is a circuit of size $\op{poly}(s')$ that produces the greatest common divisor of $f_1,\dots,f_q$. 
  \item[(DE)] Denominator Extraction: Given a reduced rational function $\frac{p(x)}{q(x)}\in k\left( x_1,\dots,x_n \right)$, produced by a circuit with division of size $s'$, there is a circuit of size $\op{poly}(s')$ that produces $q$. 
\end{itemize}
In \loccit, randomized algorithms that produce the output circuits for both DE and GCD are presented. But the circuits that output the $\op{gcd}$ and the denominator are themselves not randomized. This will be enough in our non-uniform setting; we only use the existence of the of the polynomial size circuits that produce the $\op{gcd}$ and the denominator and do not use the algorithm that produces the circuits themselves.  

Going back to the proof, we made a system of equations $A\vec{f} = 0$ from the diagram of basic morphisms and got a solution of the form $Ce_1 = \left( {p_1}/{q_1},\dots, {p_s}/{q_s} \right)^{T} \in k(x_1,\dots,x_n)^{s}$ and wanted to produce a solution in $k\left[ x_1,\dots,x_n \right]^{s}$ instead. To do this, first use Kaltofen's GCD to assume, without loss of generality that each $\frac{p_i}{q_i}$ is reduced. Then use Kaltofen's DE to extract the denominators $q_i$ from each fraction. The element $q_1q_2\dots q_s Ce_{1}$ is in $k\left[ x_1,\dots,x_n \right]^{s}$ and is a solution to $A\vec{f} = 0$. Now divide $q_1q_2\dots q_s Ce_{1}$ by the $\op{gcd}$ of all of its entries to obtain $(f_{v_1},\dots,f_{v_s}) \in k\left[ x_1,\dots,x_n \right]^{s}$. The polynomials $f_1,\dots,f_s$ correspond to the morphisms from $D$ to the colimit $R$. Each $f_{v}$ is produced by a circuit with division, but since the degrees of $f_v$ are polynomially bounded, each such circuit can be turned into a circuit without division using Strassen's method \cite{strassenvermeidung}. This concludes the proof that for any diagram of basic morphisms with $R$ as a colimit, every colimit cocone morphism from an object in $D$ sends $1$ to an $f_v$ which is computed by a circuit of size polynomial in $s$.

We now prove the theorem. Let $R \xto{1\mapsto f} R$ be a subdiagram of a colimit computation with initial step $D_0$. Call the source $R_1$ and the target $R_2$. By Lemma~\ref{lem:constructivecomputation}, there is a subdiagram $D_0'\subset D_0$ of basic morphisms whose colimit is $R_1$; and (cf.  Remark~\ref{rem:constructive-morphism-colimits}) there is a subdiagram $D_0''\subset D_0$, which contains $D_0'$, and whose colimit is $R_2$, with the induced map $R_1 \to R_2$ being a map that sends $1\mapsto f$. This implies, combined with the first part of this proof applied to both $D_0'$ and $D_0''$, that $f$ is the quotient of two polynomials computed by polynomially sized circuits. Hence, by Strassen's method, $f$ is computed by a circuit of size  polynomial in the size of $D_0$. 
\end{proof}

\begin{remark}
\label{rem:RMod}
Note that the set of sequences of polynomials having formula sizes bounded by some quasi-polynomial function (i.e. a function
of the form $n^{\log^c n}$) coincides with the sequence of polynomials having arithmetic circuit complexity bounded by some
quasi-polynomial function over any field; i.e. using the notation in \cite{Burgisser-book2} $\op{VQP}_e = \op{VQP}$ \cite[Corollary 2.27]{Burgisser-book2}. Thus, Proposition \ref{prop:formula-vs-RMod} and  Theorem \ref{thm_rmoddiagramtocircuit} together imply that
the colimit complexities of a sequence of polynomial morphisms $\left(k[x_1,\ldots,x_n] \xtonormal{1\mapsto f_n} k[x_1,\ldots,x_n]\right)_{n > 0}$ is bounded by quasi-polynomial function of $n$, if and only if,  the arithmetic circuit complexity of the sequence $(f_n)_{n > 0}$ is also bounded by
a quasi-polynomial function.
\end{remark}

\section{Functors}\label{sec:functors}
Here we discuss phenomena of categorical complexity related to functors between categories.

\subsection{Preservation under functors}\label{sec:preservationunderfunctors} 
  
Let $F: \cC \to \cD$ be a functor.  If $D$ is a diagram in $\cC$, the image diagram $F(D)$ is defined in the obvious way. If $F$ preserves finite limits, then, for every every limit computation $(D_0,\dots,D_s)$ in $\cC$, starting with a set of basic morphisms $\cA$ will correspond to a limit computation $(F(D_0),\dots,F(D_s))$ in $\cD$ which starts with basic morphisms in $F(\cA)$. Therefore, in this case, we have
    $$ c^{\op{lim}}_{\cD,F(\cA)}(F(D)) \leq c^{\op{lim}}_{\cC,\cA}(D).$$
We have the analogous statement for colimit computations if $F$ preserves colimits.  

Since equivalences preserve limits and colimits, categorical complexity, limit complexity, colimit complexity and mixed complexity are all invariant under equivalences of categories. A more general case is that of adjoint functors. Let $R:\cC\to \cD$ be a functor right-adjoint to a functor $L: \cD \to \cC$. Then, since right-adjoints preserve limits and left adjoints preserve colimits, we  have the following. 

\begin{lemma}
  \label{lem:adjointfunctors} Let $R:\cC \to \cD$ and $L:\cD \to \cC$ be a pair of adjoint functors. Let $\cA$ be a set of basic morphisms in $\cC$, and $\cA'$ be a set of basic morphisms in $\cD$, such that $R(\cA)\subset \cA'$ and $L(\cA')\subset \cA$. Then we have inequalities of complexities
  \begin{eqnarray*}
    c^{\op{lim}}_{\cD,\cA'}(R(D)) \leq c^{\op{lim}}_{\cC,\cA}(D),\\
    c^{\op{colim}}_{\cC,\cA}(L(D')) \leq c^{\op{colim}}_{\cD,\cA'}(D'),\\
  \end{eqnarray*}
  for every diagram $D$ in $\cC$ and $D'$ in $\cD$.
\end{lemma}

\begin{example}
  Let $\opcat{Vect}_{k}$ be the category of vector spaces over a field $k$. As the free vector space functor $\op{Fr}: \opcat{Set} \to \opcat{Vect}_{k}$ is left adjoint to the forgetful functor $\op{Fo} : \opcat{Vect}_{k}\to \opcat{Set}$, we have that $\op{Fr}$ preserves colimits. In particular, the 
   colimit complexity of a vector space is bounded above by its dimension. 
\end{example}

\subsection{Complexity of functors}\label{sec:complexityoffunctor}
We discussed in Section~\ref{sec:preservationunderfunctors} above, how certain functors preserve limits or colimits, and does also 
corresponding notions of complexities as well.  In this section, we will study complexities of general functors.  The first step is to define
the complexity of a functor between categories.

\begin{definition}
\label{def:abstract-complexity}
A \emph{complexity} function on a category $\cC$ is a function that takes (finite) diagrams of $\cC$ to $\mathbb{N} \cup \{\infty\}$.
\end{definition}

\begin{example}
For example, if $\cA$ is a set of morphisms of $\cC$ (assumed to have unit cost), then $c^{\op{lim}}_{\cC,\cA}, c^{\op{colim}}_{\cC,\cA}, c^{\op{mixed}}_{\cC,\cA}$ are all examples of complexity functions on $\cC$.
\end{example}

\begin{definition}[Complexity of functors]\label{def_functorcomplexity}
Let $\mathcal{C},\mathcal{D}$ be two categories with complexity functions,
$\phi, \psi$,
and  let $F: \mathcal{C} \rightarrow \mathcal{D}$ be a functor. We define the complexity, $C_{\phi,\psi}(F):\mathbb{N} \rightarrow \mathbb{N}$ by 
\begin{equation*}
  C_{\phi,\psi}(F)(n) =  \sup \left\{ 
\psi(F(D)) \,\,|\,\, I \t{ a finite directed graph}, 
\,D \in \mathcal{C}^I, 
\, \phi(D) \leq n
\right\}.
\end{equation*}
If $F$ is an endofunctor (i.e. $\cC = \cD$) and $\phi=\psi$, then we will denote $C_{\phi,\psi}(F)$ just by $C_{\phi}(F)$. 
\end{definition}

One key example of functor complexity comes from \emph{image functors}. In order to define the image functor we first need
to recall a few relevant notions from category theory.
 
\subsection{Monomorphisms, slice categories, and the image functor}
Let $\cC$ be any category. Recall the following standard definitions.

\begin{definition}[Monomorphisms]
\label{def:mono}
A morphism $f:A \to B$ in a category $\cC$ is called a \emph{monomorphism} if for all pairs of morphisms $g,h:C \to A$,
$f\compose g = f \compose h \implies  g=h$. 
\end{definition}

\begin{definition}[Slice category]
Let $A$ be an object of $\cC$. The  \emph{slice category}, $\cC/A$,  is the category whose objects are morphisms $B\to A$, and whose
morphisms are commutative diagrams:

\[
\xymatrix{
B \ar[rr]^{f}\ar[rd] &&C\ar[ld] \\
&A&
},
\]
which compose in the obvious way.

The full sub-category of $\cC/A$ consisting of monomorphisms, $f: C \rightarrowtail A$ is denoted $\Sub(A)$ (subobjects of $A$).
\end{definition}

\begin{definition}
\label{def:images}
We say that a category $\cC$ \emph{has images}, if for every object $A$ of $\cC$, the inclusion functor $\op{i}_{\cC,A}:\Sub(A) \rightarrow \cC/A$, has a left adjoint, $\op{im}_{\cC,A}$ (cf. \cite[Lemma 1.3.1]{Johnstone-vol1}).  
\hide{
This implies in particular that,
in the diagram category $\cC^{\arrowdiagram}$, letting $\op{Mon_\cC}$ denote full subcategory of monomorphisms, and $\op{i}_{\cC} : \op{Mon}_{\cC} \to \cC^{\arrowdiagram}$ the inclusion functor, that there is a left-adjoint $\op{im}_{\cC}$ to $\op{i_{\cC}}$,
}
\[
\xymatrix{
\Sub(A)  \ar[r]_{\hspace{.3in}\op{i}_{\cC,A}}  & \cC/A \ar@/_1pc/[l]_{\op{im}_{\cC,A}} \\
}.
\]
\end{definition}

\begin{remark}
Note that the adjointness $\op{im}_{\cC,A} \dashv \op{i}_{\cC,A}$ is equivalent to saying that,
given a monomorphism $f: B\rightarrowtail A$, and a morphism $g:C \to A$ in $\cC/A$,
 there is a natural isomorphism,
\[
 \cC/A(g, \op{i}_{\cC,A}(f)) \cong \Sub(A)(\op{im}_{\cC,A}(g),  f),
\]
where the isomorphism takes $\theta \in  \cC/A(g, \op{i}_{\cC,A}(f)) $ given by the diagram,

\[
\xymatrix{
C \ar[rr]^{\theta} \ar[rd]^{g} && B \ar[ld]^{f} \\
&A &
}
\]

to the element in $\Sub(A)(\op{im}_{\cC}(g),  f)$ described by the diagram
\[ 
\xymatrix{
\op{dom}(\op{im}_{\cC/A}(g)) \ar[rd]^{\op{im}_{\cC/A}(g)} \ar[rr]^{\hspace{.3in}\op{im}_{\cC/A}(\theta)} && B\ar[ld]^f \\
&A&
}.
\]
In particular, taking $f=\op{im}_{\cC,A}(g)$, the image under the inverse of the isomorphism defined above of the
element $1_{f} \in \Sub(A)$, induces a morphism $\eps(g): B \to \op{dom}(\op{im}_{\cC/A}(f))$, making the following diagram commute:
\[
\xymatrix{
B \ar@{..>}[rr]^{\hspace{-.4in}\eps(g)} \ar[rd]^g &&\op{dom}(\op{im}_{\cC/A}(f))\ar[ld]^{\op{im}_{\cC/A}(g)} \\
 &A&
 }.
 \]
\end{remark}

We now extend the image functor from the slice categories $\cC/A$, to the diagram category $\cC^{\arrowdiagram}$.
\begin{proposition}
Suppose that $\cC$ is a category that has pull-backs and images.
Then
in the diagram category $\cC^{\arrowdiagram}$, letting $\op{Mon_\cC}$ denote full subcategory of monomorphisms, and $\op{i}_{\cC} : \op{Mon}_{\cC} \to \cC^{\arrowdiagram}$ the inclusion functor, there exists a left-adjoint $\op{im}_{\cC}$ to $\op{i_{\cC}}$.
\end{proposition}

\begin{proof}
The functor $\op{im}_{\cC}$ is defined as follows. For an object  $f:C \to A$ of $\cC^{\arrowdiagram}$, we set
$\op{im}_{\cC}(f) = \op{im}_{\cC/A}(f)$. For a morphism, $\theta = (\theta_0,\theta_1)$ of 
$\cC^{\arrowdiagram}$, given by the following commutative diagram,
\[
\xymatrix{
C \ar[r]^{\theta_1}\ar[d]^f & B\ar[d]^g \\
A \ar[r]^{\theta_0} & D,
}
\]
consider the pull-back diagram,
\[
\xymatrix{
C'  \ar@{..>}[r]^{\hspace{-.4in}\theta_1'}\ar@{..>}[d] & \op{dom}(\op{im}_{\cC,B}(g))\ar[d]^{\op{im}_{\cC,B}(g)} \\
A \ar[r]^{\theta_0} & D
}.
\]
Since, the morphism $\op{im}_{\cC,B}(g)\op{dom}(\op{im}_{\cC,B}(g)) \to D$ is a monomorphism, and the pull back
of a monomorphism is again a monomorphism, the left vertical arrow is a monomorphism.
Moreover, from the universal property of the pull-back applied to the diagram,
\[
\xymatrix{
C \ar@/_/[ddr]_f \ar@/^/[drr]^{\eps(g)\circ \theta_1} \\
&   & \op{dom}(\op{im}_{\cC/B}(g))\ar[d]^{\op{im}_{\cC,B}(g)} \\
&A \ar[r]^{\theta_0} & D
},
\]
we obtain the commuting diagram
\[
\xymatrix{
C \ar@/_/[ddr]_f \ar@/^/[drr]^{\eps(g)\circ \theta_1} \ar@{.>}[dr]|-{}\\
&C'  \ar@{..>}[r]^{\hspace{-.45in}\theta_1'}\ar@{..>}[d] & \op{dom}(\op{im}_{\cC/B}(g))\ar[d]^{\op{im}_{\cC,B}(g)} \\
&A \ar[r]^{\theta_0} & D
},
\]
where the dotted diagonal arrow is isomorphic to $\eps(f):C \rightarrow \op{dom}(\op{im}_{\cC/A}(f))$ (using the fact that the left vertical arrow is a monomorphism). 
We set $\op{im}_{\cC}(\theta) = \theta' = (\theta_1',\theta_0) \in \op{Mon}_{\cC}(\op{im}_{\cC}(f),  \op{im}_{\cC}(g))$. 

To check $\op{im}_{\cC}$ is left adjoint to $\op{i}_{\cC}$ is an easy (if tedious) exercise.
\end{proof}

\begin{example}[Examples of categories with images]
Most of the categories introduced before has images (and also pullbacks). It is easy to verify that the categories  $\opcat{Vect}_k$, $\opcat{Set}$, $\opcat{Top}$, $\opcat{SL}$,
$\opcat{SA}$ have images as well as pull-backs. The fact that the last two categories 
(see Notation~\ref{not:categories}, Part \eqref{itemlabel:not:categories:SL-SA} for their definitions)
have images is a consequence of the fact the images of semi-linear
(respectively, semi-algebraic) sets are closed under taking images under affine (respectively, polynomial) maps.
\end{example}

\subsection{Complexity of the image functor}
\label{subsec:complexity-of-image-functor}
\hide{
It is a folklore result that the P vs NP question in computational complexity (say in the real B-S-S model of computation)  is very closely related to the problem of proving lower bounds on the complexity of the images under projections of sequences of semi-algebraic sets 
belonging to the class $\mathbf{P}_\R$. 
However, to make this statement precise, one needs to define complexity classes 
sets of sequences of semi-algebraic sets, and also consider sequences of projections. 
}
The P vs NP question in computational complexity (say in the real B-S-S model of computation)
is fundamentally about comparing the complexities of sequences of semi-algebraic sets which
belong to an ``easy'' class (i.e. the B-S-S complexity class $\mathrm{P}_\mathbb{R}$), with the complexities
of sequences obtained by taking images under certain projections of sequences belonging to
the ``easy'' class (by taking the images under projections of sequences in the class $\mathrm{P}_\mathbb{R}$
one obtains the B-S-S complexity class $\mathrm{NP}_{\mathbb{R}}$). A formal definition
of $\mathrm{P}_{\mathbb{R}}, \mathrm{NP}_{\mathbb{R}}$ (and indeed of the whole polynomial
hierarchy $\mathrm{PH}_\mathbb{R}$) can be found in \cite[\S 1.2.1]{BZ09} and will not be repeated here. (We only note that existential quantifiers used in the definition of $\mathrm{NP}_{\mathbb{R}}$ in \cite {BZ09} geometrically corresponds to taking image under projection, or more precisely, elimination of 
existential quantifiers geometrically corresponds taking the image under a linear projection map.)
By replacing the notion of B-S-S complexity by other non-uniform notions of complexity
one can extend these definitions to the non-uniform case as well (see \cite{Isik2019} and \cite{BP18}).
Finally, what is perhaps closest to the point of view of this paper, in \cite{BasuConstr}, 
a sheaf theoretic analog of the classes $\mathrm{P}_{\mathbb{R}}$ and $\mathrm{NP}_{\mathbb{R}}$
are  defined. The elements of these classes are no longer sequences of semi-algebraic sets, but rather of
semi-algebraically constructible sheaves, and the role of taking images under projection maps, is replaced by the direct image functor. A ``P vs NP'' conjecture in this general setting is formulated in \cite{BasuConstr}
(Conjecture   3.83). It should be clear from the above discussion that the P vs NP question is intimately related to studying how badly the complexity of certain objects (say semi-algebraic sets) blow up (i.e whether the blow-up is polynomially bounded or not) after taking the image under certain tame maps.
The notion of \  ``complexity''  of a sequence of sets in the formulations described above is different in each case, but is related for example to
the number of steps taken by a B-S-S machine, or the size of an arithmetic circuit, required  for testing membership  in the corresponding sequence of sets.

With the discussion above as background we proceed to define an analog of the P vs NP question in  the categorical setting.  The main idea is to replace ``taking images under projections'' by the image
functor in the category under question (assuming it has an image functor), and replacing the 
notion of \ 
``complexity''  by categorical complexity. 
The categorical analog turns out to be simpler, coordinate independent,
and also free of taking sequences. 

In this paper, we do not attempt in this paper to prove any formal relationships between 
the categorical and standard versions of the P vs NP question, leaving such questions for
future investigations. However, we hope the discussion in the previous paragraph makes clear
the analogy between the standard  ``P vs NP'' questions, with the categorical analog.

First observe that for any category
$\cC$, a diagram, $D:I=(V,E,s,t)\to \cC^{\arrowdiagram}$, in the diagram category of $\cC^{\arrowdiagram}$, induces a diagram 
$D':I' = (V',E',s',t') \to \cC$ as follows (cf. Figure \ref{fig:Iprime}).

\begin{figure}
\includegraphics[scale=0.5]{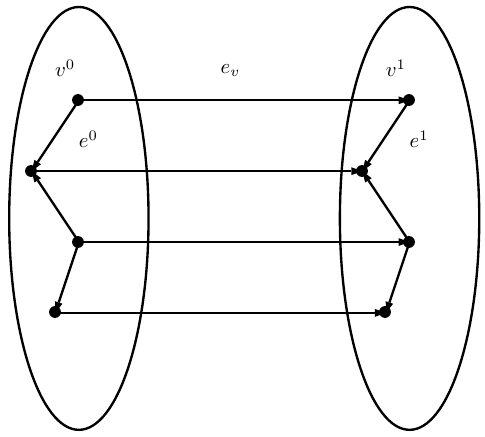}
\caption{The directed graph $I'$.}
\label{fig:Iprime}
\end{figure}
 
Here $V' = V^0 \cup V^1$ where $V^i = \{ v^i \mid v \in V\}$ is a copy of $V$ for $i=0,1$, and 
$E' = E^0 \cup E^1 \cup \bigcup_{v \in V} \{e_v\}$,
where $E^i = \{e^i \mid e \in E\}$ is a copy of $E$ for $i=0,1$, and
\[
\left.
\begin{array}{lll}
s'(e^i) &=& s(e)^i\\
t'(e^i) &=& t(e)^i
\end{array}
\right\} \mbox{ for $e \in E, i=0,1$},
\]
\[
\left.
\begin{array}{lll}
s'(e_v) &=& v^0,\\  
t'(e_v)  &=& v^1.
\end{array}
\right\} \mbox{ for $v \in V$},
\]
and 

\begin{eqnarray*}
D'(e^0) &=& D(s(e))\\
D'(e^1) &=& D(t(e)) \\
D'(e_v) &=& D(v), v \in V,
\end{eqnarray*}
remembering that for each vertex $v \in V$, $D(v)$ is a morphism of $\cC$.

Using the notation introduced above we have the following definition. 
\begin{definition} [Image complexity]
\label{def:image-complexity}
If $\cA$ is a set of morphisms in $\cC$, we will denote for $* = \op{lim},\op{colim},\op{mixed}$
for any diagram $D$  of  $\cC^{\arrowdiagram}$ (respectively, $\op{Mon}_{\cC}$), we define
the image complexity, $c^*_{\cC^{\arrowdiagram},\cA}(D)$ (respectively, $c^*_{\op{Mon}_{\cC},\cA}(D)$) of the diagram $D$ by
\begin{equation*}
c^*_{\cC^{\arrowdiagram},\cA}(D) = c^*_{\cC,\cA}(D')
\end{equation*}
(respectively, $c^*_{\op{Mon}_{\cC},\cA}(D) = c^*_{\cC,\cA}(D')$),
 where the right hand side is the (limit, colimit or mixed) complexity of the diagram $D'$  in the category $\cC$ (cf. Definition~\ref{def:categorical-complexity}),  with $\cA$ as the set of basic morphisms.
\end{definition}

\hide{
Thus, for any category $\cC$ which has images, and for any subset $\mathcal{A}$ of morphisms of $\cC$, and 
$* = \op{lim},
\op{colim},
\op{mixed}$, we can ask the following question.
}

Now suppose that a category $\cC$ has images and $\mathcal{A}$ is a set of morphisms of $\cC$.
Recall the definition of the complexity of a functor between two categories  with complexity functions (cf. Definition~\ref{def_functorcomplexity}). 
Applying Definition~\ref{def_functorcomplexity} to the functor 
\[
\op{im}_{\cC}: \cC^{\arrowdiagram} \rightarrow \op{Mon}_\C,
\]
with complexity functions $c^*_{\cC^{\arrowdiagram},\cA}$ and $c^*_{\op{Mon}_{\cC},\cA}$
(with $* = \op{lim},
\op{colim},
\op{mixed}$), on the source and target of the functor $\op{im}_{\cC}$, we obtain the function 
\[
C_{c^*_{\cC^{\arrowdiagram},\cA}, c^*_{\op{Mon}_{\cC},\cA}}(\op{im}_{\cC})(n).
\]

\begin{definition}[Complexity of the image functor]
\label{def:image-complexity}
With $* = \op{lim},
\op{colim},
\op{mixed}$,
we will denote by  
\begin{equation}
\label{eqn:image-complexity}
\IC_{\cC,\cA}^*(n) = C_{c^*_{\cC^{\arrowdiagram},\cA}, c^*_{\op{Mon}_{\cC},\cA}}(\op{im}_{\cC})(n)
\end{equation}
and call $\IC^*_{\cC,\cA}$ the \emph{$*$-complexity of the image functor} of the category $\cC$ (with respect to the
basic set of morphisms $\cA$).
\end{definition}

Now for any category $\cC$ which has images, and for any subset $\mathcal{A}$ of morphisms of $\cC$, and 
$* = \op{lim},
\op{colim},
\op{mixed}$, we can ask: 
\begin{question}
\label{question:CategoricalPvsNP}
Is the function \[
\IC^*_{\cC,\cA}
\] 
bounded by some  polynomial ? 
\end{question}

\subsection{The image functor in the categories $\opcat{SL}$ and $\opcat{SA}$}
First observe that the class of semi-linear (respectively, semi-algebraic) sets are closed under taking images of affine (respectively,
polynomial) maps. This is a consequence of Fourier-Motzkin elimination (respectively, Tarski-Seidenberg principle). However, the 
answer to the Question \ref{question:CategoricalPvsNP} may depend on the type of categorical complexity that is being considered. 

We consider the cases of 
limit, and mixed complexity, starting with the 
limit complexity in the categories 
$\opcat{SL}$ and $\opcat{SA}$. We prove that the answer to Question \ref{question:CategoricalPvsNP} in $\opcat{SL}$ and
$\opcat{SA}$ for 
limit complexity with the choice of the basic morphisms defined below.
This should not be too much of a surprise since the power of 
limit computation with these choices of basic morphisms is quite limited. Only closed convex polyhedral subsets of $\R^n$ can be constructed in $\opcat{SL}$ (respectively, basic closed semi-algebraic sets) using 
limit computations. 
(A basic closed semi-algebraic set is a semi-algebraic set defined by a conjunction of a finite number  weak polynomial inequalities 
of the form $P \geq 0$.)
Moreover, in the case of the category 
$\opcat{SL}$ we prove a polynomial upper bound on the number of facets of an object in terms of 
limit complexity (Lemma \ref{lem:facets} below). The negative answer to Question   \ref{question:CategoricalPvsNP} in this case can be deduced by exhibiting a sequence of morphisms of polynomially bounded complexity, 
such that the number of facets of the image grows super-polynomially.
We exhibit such a sequence using the well known properties of cyclic polytopes.  In the case of the category $\opcat{SA}$ the situation is even simpler, since it is well known that images under polynomial maps (in fact, projections along a coordinate) of a basic closed semi-algebraic set need not be basic closed. We exhibit a simple example of this phenomenon.

\subsubsection{Limit complexity of the image functor in  $\opcat{SL}$}
Recall that the objects of
the category $\opcat{SL}$  are (embedded) semi-linear sets and  morphisms between such sets which are restrictions of affine mappings. Let $\cA$ consist of the scalar multiplication morphisms $\R \xto{c} \R$ for each $c\in \R$, the addition morphism $\R^2 \xto{+} \R$,  
and morphisms $[0,\infty) \hookrightarrow  \R$, $\R\to 0$. It is easy to see that using  
limit computations one can produce every morphism $f:(A \subset \R^m)  \rightarrow (B \subset \R^n)$, where $A,B$ are closed polyhedral subsets of $\R^m$ and 
$\R^n$ respectively. 

\begin{lemma}
\label{lem:facets}
The number of facets of each object $A$ is bounded from above by 
$(c^{\op{lim}}_{\opcat{SL},\cA}(A))^2$.
\end{lemma}

\begin{proof}
Let $(D_0,\ldots,D_s)$ be the sequence of diagrams of a 
limit computation of $A$. We prove by an induction on $s$ the following statement from which the lemma will follow. 

For each vertex $v$ of $D_s$, let $h_s(v) = \card(H_s(v))$, where $H_s(v)$ is the set of vertices of $D_s$  which were introduced by taking a limit, and which has $v$ in its cone. We follow the convention that if $v$ is a vertex of $D_0$, then $H_s(v) = \{v \}$. 
Because of the constructivity assumption,  $h_s(v)$ is well defined. 

We claim that the number of facets of $A$ is bounded by $\sum_{v \in D_s} h_s(v)$. 
We prove the claim by an induction on $s$. The claim is clearly true if $s=0$, since the objects which are domains or codomains
of the basic morphisms have at most one facet, and $h_s(v) = 1$ for each vertex of $D_0$. Now, assume that the claim holds
for all $s' < s$.

For each vertex $v$ of $D_{s-1}$,  let $A_v = D_{s-1}(v)$. 
Let $V_{s-1}$ be the set of vertices of $D_{s-1}$ which do not belong to the limit cone of any vertex
$v$ of $D_{s-1}$ not equal to itself. Then, $A$ is the intersection of the product of the polytopes $A_v, v \in V_{s-1}$ with an affine subspace. This implies that the number of facets of $A$ is bounded by the sum of the number of facets of $A_v, v \in V_{s-1}$.
Now use the induction hypothesis to finish the proof of the claim.

It is clear that the lemma follows from the claim since for each vertex $v$ of $D_s$, $h_s(v)$ is bounded by the total number of vertices of $D_s$, which itself is bounded by the cost of the diagram computation  $(D_0,\ldots,D_s)$.
\end{proof}

Consider now the following example.
Denote by $\theta_m(t) = (t,t^2,\ldots,t^{2m}) \subset \R^{2m}$, the moment curve in $\R^{2m}$. The convex hull of $n$ distinct
points on the moment curve is called a cyclic polytope, $\op{Cyl}(n,m)$. For $ n> 2m$, the number of facets of $\op{Cyl}(n,m)$ is given by
\cite[Theorem 4]{Gale}
\[
\binom{n-m}{m} + \binom{n-m-1}{m-1}.
\]
In particular, the number of facets of $\op{Cyl}(4m+1,2m)$ equals $4m$. It is easy to derive that the 
limit complexity of 
$\op{Cyl}(4m+1,2m)$ is bounded by $O(m^2)$.
Now let $\pi_m:\R^{4m} \to \R^{2m}$ be the projection on the first $2m$ coordinates. Clearly, $\pi_m(\op{Cyl}(4m+1,2m)) = \op{Cyl}(4m+1,m)$,
and the number of facets of   $\op{Cyl}(4m+1,m)$ equals 
\[
\binom{3m+1}{m} + \binom{4m}{m-1}
\]
which is clearly exponential in $m$. 
It follows from Lemma \ref{lem:facets} that
the 
limit complexity of $\pi_m(\op{Cyl}(4m+1,2m))$ is exponentially large in $m$. 

Thus, the object $\pi_m: \op{Cyl}(4m+1,2m) \to \R^{2m}$ of the category $\opcat{SL}^{\arrowdiagram}$ has  
limit complexity polynomially bounded in $m$,
the 
limit complexity of $\op(im_{\opcat{SL}})(\pi_m)$ is exponentially large in $m$. 
It follows from this example that:
\begin{proposition}
\label{prop:IC-SL}
The function 
$\IC^{\lim}_{\opcat{SL},\cA}$ is not polynomially bounded
(where
$\cA$ consists of the scalar multiplication morphisms $\R \xto{c} \R$ for each $c\in \R$, the addition morphism $\R^2 \xto{+} \R$,  
and morphisms $[0,\infty) \hookrightarrow  \R$, $\R\to 0$).
\end{proposition}

\subsubsection{
Limit complexity of the image functor  in $\opcat{SA}$}
The case of the category $\opcat{SA}$ with respect to 
limit complexity is simpler. 
Recall that the objects of
the category $\opcat{SA}$  are (embedded) semi-algebraic sets and  morphisms between such sets which are restrictions of polynomial mappings. Let $\cA$ consist of the scalar multiplication morphisms $\R \xto{c} \R$ for each $c\in \R$, the addition morphism $\R^2 \xto{+} \R$,  
the multiplication  morphism $\R^2 \xto{\cdot} \R$,
and morphisms $[0,\infty) \hookrightarrow  \R$, $\R\to 0$. It is not difficult to see that the objects of $\opcat{SA}$ that can be constructed using a 
limit computation are exactly the basic closed semi-algebraic sets. On the other hand, it is well known that the image under polynomial maps
(for example, projections along some coordinates) of a basic closed semi-algebraic set need not be a basic closed semi-algebraic set.
For example, consider the real variety $V$ defined by 
\[
(X_1 - X_3^2)(X_2 - X_4^2) = 0.
\]
Denoting by $\pi:\R^4 \to \R^2$ the projection to $X_1,X_2$-coordinates, 
$\pi(V) = \{ (x_1,x_2) \in \R^2 \mid x_1 \geq 0 \vee x_2 \geq 0\}$ which is not a basic closed semi-algebraic set (as observed by Lojasiewicz, 
see  \cite[page 466]{Andradas-Ruiz}), and hence $\pi(V)$ has infinite 
limit complexity.

It follows that:

\begin{proposition}
\label{prop:IC-SA}
The function 
$\IC^{\lim}_{\opcat{SA},\cA}$
is not polynomially bounded (where
$\cA$ consists of the scalar multiplication morphisms $\R \xto{c} \R$ for each $c\in \R$, the addition morphism $\R^2 \xto{+} \R$,  
the multiplication  morphism $\R^2 \xto{\cdot} \R$,
and morphisms $[0,\infty) \hookrightarrow  \R$, $\R\to 0$).
\end{proposition}

\subsubsection{Complexity of the image functor with respect to mixed complexity in $\opcat{SL}$ and $\opcat{SA}$}
Mixed computation is  much more powerful than limit computation and we are unable to resolve the following question
which we posit as the categorical analog of the P vs NP question in the categories $\opcat{SL}$ and $\opcat{SA}$  (refer to the discussion in the beginning of this section).

\begin{question}
\label{question:mixed-SL-SA}
Are the functions 
$\IC^{\op{mixed}}_{\opcat{SL},\cA}, \IC^{\op{mixed}}_{\opcat{SA},\cA} $ 
polynomially bounded  (with the set of morphisms $\cA$ being the same as in Propositions~\ref{prop:IC-SL} and \ref{prop:IC-SA} respectively)
?
\end{question}

We also make the following conjecture.

\begin{conjecture}
The function 
$\IC^{\op{mixed}}_{\opcat{SA},\cA}$ 
is bounded singly exponentially.
\end{conjecture}
This should be thought of as the categorical analog of the existence of algorithms with singly exponential algorithms for elimination
of one block of existential quantifiers in the first order theory of the reals (see for example \cite{BPRbook2}).

\subsection{The image functor in the category $k\left[ x_1,\dots,x_n \right]\Mod$ and colimit complexity}
We have considered complexity of the image functor with respect to limit and mixed complexity in several categories. 
We now consider the image functor from the point of view of colimit complexity in the category 
$R\Mod$, 
where $R = k\left[ x_1,\dots x_n \right]$,
with the set basic morphisms defined in 
\eqref{eqn:basic-RMod}.

We describe a method for writing a colimit computation to compute the image of morphism of modules. A careful analysis of the part of the proof of Proposition
\ref{prop_rmod_images} which makes use of  the existence of 
Gr\"obner basis of modules would give an upper bound for the complexity of the image functor in $R\Mod$. With a naive analysis, we can only obtain an unnecessarily high bound (doubly exponential), so we leave the complexity analysis 
(that is the explicit dependence on the parameter $s$)
 out of the following statement.  

\begin{proposition}\label{prop_rmod_images}
  Let $M\xto{\varphi} N$ be a morphism diagram in $R\Mod$, computed by a colimit computation
 of size $s$. 
  Then, there is a colimit computation that computes $\op{im}(\varphi) \to N$ 
 of size bounded  by some function of $s$.
\end{proposition}
\begin{proof}
  By Lemma~\ref{lem:constructivecomputation} and Remark~\ref{rem:constructive-morphism-colimits}, the colimit computation that produces $M \xto{\varphi} N$ can be simplified to produce (only) $M \xto{\varphi} N$ in two colimit steps. The first one is taking the colimit of diagram $D_{M}$ of basic morphisms which produces $M$. The second one is the colimit of a larger diagram $D_N \supset D_M$ together with $M$ and all the cocone morphisms from the objects in $D_M$ to $M$. Replacing colimits with coequalizers and coproducts, we have that $M$ is isomorphic to the quotient
\[
\xymatrix{
\Bigoplus_{\rho \in s(D_M)} R^{j_{\rho}}  \ar[r]^{A_M}  & \Bigoplus_{\gamma \in \op{v}(D_M)} R^{j_{\gamma}} \ar[r] & M,
}
\]
where $\gamma$ runs over all the vertices in $D_M$ and $\rho$ runs over all sources of arrows in $D_M$; and $j_{\rho}$ and $j_{\gamma}$ are $0$, $1$ or $2$, based on whether the corresponding basic object is $\left\{ 0 \right\}$, $R$, or $R^2$. 
(Note that when using Remark~\ref{rem:constructive-morphism-colimits} we are implicitly using the fact, that in the  category $R\Mod$, the coequalizer of two morphisms  $\phi,\psi: A \rightarrow B$ is isomorphic to the cokernel of $\phi - \psi$.)
 
Similarly we can write $N$ as a quotient,  
\begin{equation}
\xymatrixcolsep{2pc}
\xymatrix{
\Bigoplus_{\rho \in s(D_N)} R^{j_{\rho}} \oplus \Bigoplus_{v \in\op{v}(D_M)} R^{j_{v}} \ar[r]^>>>>>{\,}  &\Bigoplus_{\gamma \in \op{v}(D_N)} R^{j_\gamma} \oplus M \ar[r]  &  N
}.
  \label{eq_NNN}
\end{equation}

Since $M$ is the colimit of $D_M$, we can remove the extra sum on the left and the $M$ summand in middle. So $N$ is the quotient: 
\[
\xymatrix{
\Bigoplus_{\rho \in s(D_N)} R^{j_{\rho}} \ar[r]^>>>>>{A}  & \Bigoplus_{\gamma \in \op{v}(D_N)} R^{j_\gamma} \ar[r]  &  N
}
\]
Combining the direct sums, we get a commuting diagram with exact rows :
\[
\xymatrix{
R^{m_2}  \ar[r]^{A_M} \ar[d]^{i_2}  &  R^{m_{1}} \ar[r]\ar[d]^{i_1} & M\ar[d]^{\varphi}\ar[r] &0\\
R^{n_2} \ar[r]^>>>>>{A}  &  R^{n_1} \ar[r]  &  N \ar[r]&0,
}
\]
where the maps $i_1$ and $i_2$ are the inclusion of the first $m_1$ and $m_2$ coordinate spaces respectively. Written this way, the map $\varphi$ from $M$ to $N$ is induced by the inclusion of the generators of $M$ into the generators of $N$. 

To construct the image $\op{im}(\varphi)$, we need to find the relations among the images of the generators of $M$, i.e. the first $m_1$ generators of $N$. It is possible to use Gr\"obner basis methods to obtain a presentation of $\op{im}(\varphi)$,
that is obtain a $B = m_1 \times u$ matrix with entries in $R$, such that  
\[
\xymatrix{
R^u\ar[r]^{B} & R^{m_1}\ar[d]^{i_1}\ar[r] & \op{im}{\varphi}\ar[r]\ar[d]_{\iota}&0 \\
R^{n_2}\ar[r] & R^{m_2} \ar[r]& N \ar[r] & 0
} 
\]
(see for instance \cite[Proposition 3.3.1, Part (b)]{Kreuzer-Robbiano}).
The size of the matrix $B$ and the degrees of the polynomials appearing in it are bounded by some explicit  
(possibly doubly exponential or worse) function of the the degrees of the polynomials appearing in the matrix $A$
coming from the analysis of algorithms for computing Gr\"obner basis of submodules of a free $R$-module given a set
of generators. The important point for us is that these degrees and the size of $B$ are thus bounded by some function of
$s$ (the colimit complexity of the given morphism $M\xto{\varphi} N$).  

It is easy to see that using the basic morphisms and colimits one can realize using a colimit computation any $R\mod$ homomorphisms between between
two free $R$-modules of finite rank, with cost
depending on the degrees of the polynomials appearing in the matrix  corresponding to the homomorphism. 
Thus, we can obtain the homomorphism 
$\R^u \xtonormal{B} \R^{m_1}$ using colimit computation of size bounded by some function of $s$.

We can now obtain the $R$-module  $\op{im}(\varphi)$ using a colimit computation by taking the colimit
of the following diagram:
\[
0 \leftarrow R^{u} \xto{B} R^{m_1}.
\]
The inclusion homomorphism $\iota: \op{im}(\varphi) \xto{\iota} N$ can be computed, by
constructing $N$ again by taking the basic object $R$'s in the computation of $\op{im}(\varphi)$ corresponding to the generators of $M$ and constructing any remaining generators and relations for $N$ from \eqref{eq_NNN} above.

\end{proof}

\section{Open problems and future directions}
\label{sec:open}
In this section, we list some open problems suggested by the contents of this paper.
\begin{enumerate}[1.]
\item
Resolve questions \ref{question:permanent} and \ref{question:gap}.
\item 
Prove singly exponential upper bound on the mixed complexity of the image functor in the semi-algebraic category. This would correspond to the singly exponential complexity
algorithms for computing the image of semi-algebraic sets under polynomial maps, that follow from the critical point method in algorithmic real algebraic geometry (see for example 
\cite[Chapter 14]{BPRbook2}).
 
\item Prove singly exponential upper bound on the categorical complexity of the image functor in the category $R\Mod$ (cf. Proposition \ref{prop_rmod_images}).
\item
One important research direction in quantitative semi-algebraic geometry has been to prove tight bounds on the topological complexity of semi-algebraic sets in terms of various
parameters bounding the sizes of the formulas defining them -- such as the number of polynomial inequalities, their degrees or the number of monomials in their support or even
the additive complexity (see for example \cite{BPR10}). One could ask for similar results in the categorical setting. In the language introduced in this paper this would mean
investigating the \emph{complexity of the homology functor  $\HH: \opcat{SA} \rightarrow \Z\Mod$}. 

\item Investigate the relationship between the B-S-S complexity classes over the real numbers and  the (mixed)-categorical
complexity of objects in $\opcat{SA}$.

\item
Prove or disprove that the image functor in various categories has polynomially bounded (mixed) complexity. This is the categorical version of the P vs NP question in 
ordinary complexity theory (see Questions \ref{question:CategoricalPvsNP}, \ref{question:mixed-SL-SA}).

\item Notice that the image functor $\op{im}_{\cC}$ has a right adjoint $\op{i}_{\cC}$ whose complexity  is clearly bounded polynomially 
(in fact, $C_{\op{i}_{\cC},\cA}(n) = n$ for any set $\cA$ of basic morphisms  in $\cC$). More generally, functors often come in adjoint pairs, and often it is easy to show that one of them has polynomially bounded complexity, while the other is conjecturally hard (i.e. not polynomially bounded). More examples are given in the category of constructible sheaves  in \cite{BasuConstr}. As already noted left adjoint functors preserve colimits and thus are well behaved with respect to colimit computations, while right adjoint ones preserve limits and are well behaved with respect to limit computations. It will be interesting to analyze the role of adjointness on mixed limit-colimit computations in different categories.

\item
Extend the notion of categorical and functor complexity to categories which are further enriched -- for example, triangulated categories, derived categories etc. This would
enable one to study for example the complexity of sheaves in the derived category and various natural functors (for example the six operations of Grothendieck) from the point of view of categorical complexity. A first attempt towards studying the complexity of constructible sheaves and their functors was undertaken in \cite{BasuConstr}. However, the definition 
of complexity in \cite{BasuConstr} was not categorical.

\end{enumerate}

\bibliographystyle{amsalpha}
\bibliography{master}

\end{document}